\documentclass[oneside]{amsart}

\usepackage{graphicx}
 \usepackage{mathptmx}     
 
\usepackage{latexsym}
\usepackage{amsmath,amsrefs,amssymb,color}
\usepackage{geometry} 
\usepackage[colorlinks=true, pdfborder={0 0 0}]{hyperref}
\usepackage{upref}

\newcommand{\R}{\mathbb{R}}

\renewcommand{\[}{\left[}
\renewcommand{\]}{\right]}
\renewcommand{\(}{\left(}
\renewcommand{\)}{\right)}

\newcommand{\crits}{2^\star(s)}
\newcommand{\crit}{2^\star}
\newcommand{\rr}{\mathbb{R}}
\newtheorem{step}{Step}[section]
\newtheorem{theorem}{Theorem}[section]

\newtheorem{proposition}[theorem]{Proposition}
\newtheorem{lemma}[theorem]{Lemma}
\newtheorem{cor}[theorem]{Corollary}

\newtheorem{remark}[theorem]{Remark}
\title[Sharp profiles of singular solutions to elliptic equations: Examples]{Existence of sharp asymptotic profiles of singular solutions to an elliptic equation with a sign-changing non-linearity}

\date{December 14, 2018}
\author{Florica C. C\^irstea}

\address{Florica C. C\^irstea, School of Mathematics and Statistics, The University of Sydney, NSW 2006, Australia} 
\email{florica.cirstea@sydney.edu.au}

\author{Fr\'ed\'eric Robert}

\address{Fr\'ed\'eric Robert, Institut \'Elie Cartan, Universit\'e de Lorraine, BP 70239, F-54506 Vand{\oe}uvre-l\`es-Nancy, France}

\email{frederic.robert@univ-lorraine.fr}

\author{J\'er\^ome V\'etois}

\address{J\'er\^ome V\'etois, McGill University, Department of Mathematics and Statistics, 805 Sherbrooke Street West, Montreal, Quebec H3A 0B9, Canada.}
\email{jerome.vetois@mcgill.ca}
\begin{document}

\begin{abstract}
The first two authors [Proc. Lond. Math. Soc. (3) {\bf 114}(1):1--34, 2017] classified the behaviour near zero for all positive solutions of the perturbed elliptic equation with a critical Hardy--Sobolev growth
	$$-\Delta u=|x|^{-s} u^{\crits-1} -\mu u^q \hbox{ in }B\setminus\{0\},$$
where $B$ denotes the open unit ball centred at $0$ in $\R^n$ for $n\geq 3$, $s\in (0,2)$, $2^\star(s):=2(n-s)/(n-2)$, $\mu>0$ and $q>1$. 
For $q\in (1,\crit-1)$ with $\crit=2n/(n-2)$, it was shown in the op. cit. that the positive solutions with a non-removable singularity at $0$ 
could exhibit up to three different singular profiles, although their existence was left open. In the present paper, we settle this question for all three singular profiles in the maximal possible range. 
As an important novelty for $\mu>0$, we prove that for every $q\in (\crits -1,\crit-1)$ there exist infinitely many 
positive solutions satisfying $|x|^{s/(q-\crits+1)}u(x)\to \mu^{-1/(q-\crits+1)}$ as $|x|\to 0$, using a dynamical system approach. Moreover, 
we show that there exists a
positive singular solution with $\liminf_{|x|\to 0} |x|^{(n-2)/2} u(x)=0$ and 
 $\limsup_{|x|\to 0} |x|^{(n-2)/2} u(x)\in (0,\infty)$ if (and only if) $q\in (\crit-2,\crit-1)$. 
\end{abstract}

\maketitle

\section{Introduction and main results}\label{Sec1}
The Hardy--Sobolev inequality is obtained by interpolating between the Sobolev inequality ($s=0$) and the Hardy inequality
($s=2$): For every $s\in (0,2)$ and $n\geq 3$, there exists a positive constant $K_{s,n}$ such that
$$ \int_{\R^n} |\nabla u|^2\,dx \geq  K_{s,n} \left( \int_{\R^n} |x|^{-s} |u|^{\crits} \,dx \right)^{\frac{2}{\crits}} \quad \text{for all } u\in C^\infty_c(\R^n),
$$
where $ \crits:=2(n-s)/(n-2)$ denotes the critical Hardy--Sobolev exponent. 
The critical Sobolev exponent $\crit$ corresponds to $\crits$ with $s=0$. 
Recent results and challenges on the Hardy--Sobolev inequalities are surveyed by Ghoussoub--Robert in \cite{GR1}, see also \cite{GR3}. For $s\in (0,2)$, the best Hardy--Sobolev constant $K_{s,n}$ is attained by a one-parameter family $(U_\eta)_{\eta>0}$ 
of functions 
\begin{equation} \label{ulamb} U_\eta(x):= c_{n,s} \,\eta^{\frac{n-2}{2}} \left( \eta^{2-s}+|x|^{2-s}\right)^{-\frac{n-2}{2-s}}\quad \text{for } x\in\R^n, \end{equation} where $c_{n,s}:=\left((n-s)(n-2)\right)^{1/(\crits-2)}$ is a positive normalising constant. The functions $U_\eta$ are the only positive {\em non-singular} solutions of the equation (see Chen--Lin \cite{ChenLin} and Chou--Chu \cite{cc})
\begin{equation} \label{rnsol} -\Delta U=|x|^{-s} U^{\crits-1} \quad \text{in  } \R^n\setminus\{0\}. \end{equation} 
Moreover, any positive $C^2(\R^n\setminus\{0\})$ {\em singular} solution $U$ of \eqref{rnsol} is radially symmetric around $0$ and $v(t)=e^{-(n-2)t/2} U(e^{-t})$ is a positive periodic function of $t$ in $\R$ (see Hsia--Lin--Wang \cite{HLW}). 

The isolated singularity problem has been studied extensively, see V\'eron's monograph \cite{Ve}. Recent works of the first author and her collaborators such as \cites{CCi,FC,CirRob} give a full classification of the isolated singularities for various classes of elliptic equations. 

In this paper, we settle an open question arising from \cite{CirRob} with regard to the {\em existence} of all the singular profiles at zero for the positive solutions of the perturbed non-linear elliptic equation
\begin{equation}\label{Eq0}
-\Delta u=|x|^{-s} u^{\crits-1}-\mu u^q \qquad \text{for } x\in B(0,R)\setminus\{0\},\end{equation}
where $\mu$ is a {\em positive} parameter, $q>1$ and 
$s\in (0,2)$. By $B(0,R)$ we denote the open ball in $\R^n$ $(n\geq 3)$ centred at $0$ with radius $R>0$. 
The first two authors have proved in \cite{CirRob} that 
 the positive singular solutions of \eqref{Eq0} can exhibit up to 
{\em three} types of singular profiles at zero in a suitable range for $q$: 

$\bullet$ A {\bf (ND) type }profile (for ``Non Differential") if 
$$\lim_{|x|\to 0}|x|^{\frac{s}{q-(\crits-1)}}u(x)=\mu^{-\frac{1}{q-(\crits-1)}}.\eqno{(ND)}$$

$\bullet$ A profile of {\bf (MB) type} (for ``Multi-Bump") in the sense that there exists a sequence $(r_k)_{k\geq 0}$ of positive numbers decreasing to $0$ such that $r_{k+1}=o(r_k)$ as $k\to +\infty$ and 
\begin{equation*}
u(x)=\left(1+o(1)\right)\sum_{k=0}^\infty U_{r_k}(x)
\hbox{ as }|x|\to 0,\hbox{ where }U_\eta\hbox{ is as in } \eqref{ulamb}.\eqno{(MB)}
\end{equation*}

$\bullet$ A profile of {\bf (CGS) type}  (for ``Caffarelli--Gidas--Spruck") if there exists a positive periodic function $v\in C^\infty(\R)$ such that
$$\lim_{|x|\to 0}\left(|x|^{\frac{n-2}{2}}u(x)-v(-\log  |x|)\right)=0.\eqno{(CGS)}$$

\noindent The case $q=\crit-1$ in \eqref{Eq0} was fully dealt with in \cite{CirRob}. Hence, in the sequel we assume that $q\not=\crit-1$. 
We recall the relevant classification result from \cite{CirRob}:
\begin{theorem}[\cite{CirRob}]\label{thm:1} Let $u\in C^\infty(B(0,R)\setminus\{0\})$ be an arbitrary positive solution to \eqref{Eq0}. 
 \begin{itemize}
\item If $q>\crit-1$, then $0$ is a removable singularity;
\item If $\crits-1<q<\crit-1$, then either $0$ is a removable singularity, or $u$ develops a profile of type (CGS), (MB) or (ND);
\item If $1<q\leq \crits-1$, then either $0$ is a removable singularity, or $u$ has a profile of type (CGS) or (MB).
\end{itemize}
Moreover, if $u$ develops a profile of (MB) type, then $2^\star-2<q<2^\star-1$. 
\end{theorem}

However, no examples of the three singular profiles of Theorem~\ref{thm:1} were given in \cite{CirRob}, leaving open the question of their existence. In the present paper, we fill this gap by proving the following: 

\begin{theorem}\label{Th0} The three singular profiles of Theorem \ref{thm:1} actually do exist.
\end{theorem}

The existence assertion of Theorem \ref{Th0} is a corollary of the following precise result:
\begin{theorem}\label{Th1} Equation \eqref{Eq0} admits positive radially symmetric solutions developing  (CGS), (MB) and (ND) profiles in the exact range of parameters given by Theorem \ref{thm:1}. More precisely, when $q\in\(1,2^\star-1\)$,  
there exists $R_0>0$ such that for every $R\in\(0,R_0\)$, the following hold:
\begin{itemize}
\item[{\rm (i)}] For every $\gamma>0$, there exists a unique positive radial solution 
$u_\gamma$ of \eqref{Eq0} with a removable singularity at $0$ and 
$\lim_{|x|\to 0}u_\gamma\(x\)=\gamma$. 
\item[{\rm (ii)}] If $q>2^\star-2$, then \eqref{Eq0} has at least a positive (MB) solution. 
\item[{\rm (iii)}] For every positive singular solution $U$ of \eqref{rnsol},  
there exists a unique positive radial (CGS) solution $u$ of \eqref{Eq0} with asymptotic profile $U$ near zero.
\item[{\rm (iv)}] If $q>2^\star\(s\)-1$, then \eqref{Eq0} admits infinitely many positive (ND) solutions.
\end{itemize}
\end{theorem}

\begin{remark} If $q\in (1,\crits-1)$, then all positive radial solutions of \eqref{Eq0} extend as positive radial solutions in $\R^n\setminus\{0\}$. For $ q\in [\crits-1, \crit-1)$, any positive radial non-(ND) solution $u$ of \eqref{Eq0} extends as a positive radial solution at least in $B(0,R^*)\setminus\{0\}$ with $R^*$ independent of $u$ (see Lemma~\ref{pos}).
\end{remark}

From the three singular profiles of \eqref{Eq0}, only the (CGS) type is reminiscent of the asymptotics of the local singular solutions for the Yamabe problem 
in the case of a flat background metric ($\mu=s=0$) 
studied in Caffarelli--Gidas--Spruck \cite{CGS} (see also Korevaar--Mazzeo--Pacard--Schoen \cite{KM} for a refined asymptotics and 
Marques \cite{Ma} for the case of a general background metric).  
But for $\mu>0$, the introduction of the perturbation term in \eqref{Eq0} yields two new singular profiles: the 
(ND) and (MB) types.   

An important novelty in this paper is the {\em existence of infinitely many} positive radial (ND) solutions for \eqref{Eq0} when $q\in (\crits-1,\crit-1)$.  
To our best knowledge, there are no previous existence results known for this type of singularities, which arise as a consequence of studying \eqref{Eq0} with a critical Hardy--Sobolev growth (i.e., $s\in (0,2)$) rather than with a  critical Sobolev growth ($s=0$).    
Since \eqref{upli} fails for the (ND) solutions, neither Pohozaev-type arguments nor Fowler-type transformations relevant for (CGS) or (MB) profiles can be used. Specific to the (ND) solutions, the {\em first term} in their asymptotics arises from the competition generated in the right-hand side of \eqref{Eq0} and not directly from the differential structure. To overcome this obstacle, we rewrite the radial form of \eqref{Eq0} as a dynamical system 
using an original  transformation involving {\em three} variables, see \eqref{var3}. The variable $X_1$ in \eqref{var3} is suggestive of a second order term in the asymptotics of the (ND) solutions, which will make apparent the differential structure of our equation in a dynamical systems setting. Nevertheless, by linearising the flow around the critical point, we find a positive eigenvalue, a null one and a negative eigenvalue so that we cannot apply the classical Hartman--Grobman theorem. Instead, we shall use Theorem~\ref{71} in the Appendix, which invokes the notion of center-stable manifold and ideas of Kelley \cite{Kel}. 

For $1<q<\crit-1$, Theorem~\ref{thm:1} yields that every positive non-(ND) solution of \eqref{Eq0} satisfies
\begin{equation} \label{upli} \limsup_{|x|\to 0} |x|^{\frac{n-2}{2}} u(x)<\infty.
\end{equation}
Moreover, \eqref{upli} holds for every positive solution of \eqref{Eq0} when $q\in (1,2^\star(s)-1]$. 
Note that \eqref{upli} is crucial for  
Pohozaev type arguments \cite{CirRob}, on the basis of which we prove in Sect.~\ref{sec-thm2} the non-existence of smooth positive solutions for \eqref{Eq0}, subject to $u=0$ on $\partial B(0,R)$. 

\begin{theorem} \label{thm2} Let $\mu>0$ and $s\in (0,2)$ be arbitrary. Let $\Omega$ be a smooth bounded domain in $\R^n$ ($n\geq 3$) such that $0\in \Omega$.
Assume that $\Omega$ is star-shaped with respect to $0$. Then, for every $q\in (1,\crits-1]$, there are no positive smooth solutions for the problem
\begin{equation} \label{diri} 
\left\{ \begin{aligned}
& -\Delta u=|x|^{-s} u^{\crits-1}-\mu u^q && \text{in }  \Omega\setminus\{0\},&\\
& u=0 && \text{on } \partial \Omega.&
\end{aligned} \right.
\end{equation}  
If $ q\in (\crits-1, \crit-1)$, then \eqref{diri} admits no positive smooth solutions of non-(ND) type. 
\end{theorem}

Motivated by the problem of finding a metric conformal to the flat metric of $\R^n$ such that $K(x)$ is the scalar curvature of the new metric, 
Chen--Lin \cites{CLn1,CLn2,ChenLin} and Lin \cite{Lin} analysed the local behaviour of the positive singular solutions
$u\in C^2(B(0,1)\setminus\{0\})$ to  
\begin{equation} \label{ccli} -\Delta u=K(x) \,u^{\crit-1}\quad \text{in } B(0,1)\setminus\{0\},\end{equation}
where $K$ is a positive continuous function on $B(0,1)$ in $\R^n$ ($n\geq 3$) with $K(0)=1$.  
Moreover, $K$ was always assumed to be a $C^1$ function on $B(0,1)\setminus\{0\}$ such that
\begin{equation} \label{gag} 0<\underline{L}:=\liminf_{|x|\to 0} |x|^{1-\ell} |\nabla K(x)|\leq  
\overline{L}:=\limsup_{|x|\to 0} |x|^{1-\ell} |\nabla K(x)|<\infty\ \  \text{for some } \ell>0. 
\end{equation}
In the above-mentioned works (see also Lin--Prajapat \cite{LinPra} and Taliaferro--Zhang \cite{TaZa}), the following question was investigated: 
{\em Under what conditions on $K$, the positive singular solutions of 
\eqref{rnsol} with $s=0$ are asymptotic models at zero for the positive singular solutions of \eqref{ccli}?}  

This question was settled positively 
in any of the following situations: 
\begin{enumerate} 
	\item[(a)] Assumption \eqref{gag} holds for $\ell\geq (n-2)/2$ (see \cite{ChenLin}*{Theorems 1.1 and 1.2}); 
	\item[(b)] If \eqref{gag} holds with $\ell\in (0,(n-2)/2)$, together with {\em extra} conditions, see \cite{Lin}*{Theorem~1.2}. 
\end{enumerate}

Extra conditions in situation (b) are needed to guarantee a positive answer to the above question. Otherwise,  
for every $0<\ell<(n-2)/2$, Chen--Lin \cite{ChenLin}*{Theorem 1.6} provided 
general positive radial functions $K(r)$ non-increasing in $r=|x|\in [0,1]$ with $K(0)=1$ such that \eqref{gag} holds and 
\eqref{ccli}
has a positive singular solution with $\liminf_{|x|\to 0} |x|^{(n-2)/2}u(x)=0$.

The importance of condition \eqref{gag} in settling the above question can be inferred from our next result as a by-product of Theorem~\ref{Th1}(ii):    
For every $0<\ell<\min\{(n-2)/2,2\}$ and $s\in (0,2)\setminus\{\ell\}$, we construct a positive continuous function $K$ on $B(0,R)$ for some $R>0$ with $K(0)=1$ such that {\em exactly one inequality in \eqref{gag} fails}, yet generating for \eqref{sgt}  
a positive singular solution,      
the asymptotics of which at zero {\em cannot} be modelled by any positive singular solution of
\eqref{rnsol}.

\begin{cor} \label{corol} For every $0<\ell<\min\{(n-2)/2,2\} $ and $s\in (0,2)\setminus \{\ell\}$, there exist $R>0$ and a positive $C^1$-function
$K$ on $B(0,R)\setminus\{0\}$ in $\R^n$ ($n\geq 3$) with $K<\lim_{|x|\to 0} K(x)=1$ on $B(0,R)\setminus\{0\}$
such that $0=\underline L<\overline L<\infty$ if $\ell<s$ and
$0<\underline L<\overline L=\infty$ if $\ell>s$, yet 
\begin{equation} \label{sgt} -\Delta u =K(x) |x|^{-s} u^{\crits-1} \quad \text{in } B(0,R)\setminus \{0\} \end{equation}
admits a positive singular solution with $\liminf_{|x|\to 0} |x|^{(n-2)/2}u(x)=0$.
 \end{cor}

\paragraph{Structure of the paper.}
In Sect.~\ref{Sec6}, we prove Theorem~\ref{Th1}(iv) on the existence of infinitely many positive (ND) solutions for \eqref{Eq0}. 
In Sect.~\ref{sec-thm2},  
we establish Theorem~\ref{thm2}, together with uniform {\em a priori} estimates for the positive radial solutions of \eqref{Eq0} satisfying \eqref{upli} (see Proposition~\ref{Pr}). 
In Sect.~\ref{Sec3},  
by setting $u(r)=y(\xi)$ with $\xi=r^{(2-s)/2}$, we reduce the assertion of Theorem~\ref{Th1}(i) on removable singularities to 
the existence and uniqueness of the solution for \eqref{rem} on an interval $[0,T]$. The latter follows from Biles--Robinson--Spraker \cite{BRS}*{Theorems~1 and 2}.    
In Sect.~\ref{Sec4}, after giving the proof of Corollary~\ref{corol}, we use an argument influenced by Chen--Lin \cite{ChenLin} to prove the existence of (MB) solutions for \eqref{Eq0} in the whole possible range $q \in (\crit-2,\crit-1)$.  
In Sect.~\ref{Sec5}, with a dynamical system approach, we prove Theorem~\ref{Th1}(iii): 
the positive singular solutions of \eqref{rnsol} serve as asymptotic models for the positive radial (CGS) solutions of \eqref{Eq0}. 
For a dynamical approach to Emden--Fowler equations and systems, see Bidaut-V\'eron--Giacomini \cite{BG}. 

The results in this paper give the existence and profile at infinity for the positive solutions to 
$$ -\Delta \tilde u=|x|^{-s} \tilde u^{\crits -1} -\mu |x|^{(n-2)q-(n+2)} \tilde u^q\quad \text{for } |x|>1/R $$
by using the Kelvin transform $\tilde u(x)=|x|^{2-n}u(x/|x|^2)$, 
where $u$ is a positive solution of \eqref{Eq0}. 

\section{(ND) solutions}\label{Sec6}

In this section, we let $q\in (2^\star(s)-1,2^\star-1)$ and prove Theorem~\ref{Th1}(iv), restated below. 

\begin{proposition} \label{NDproof} 
Assume that $q\in (2^\star(s)-1,2^\star-1)$. Then, there exists $R_0>0$ such that for every $R\in (0,R_0)$, equation \eqref{Eq0} admits infinitely many positive (ND) solutions. 
\end{proposition}

The proof of Proposition~\ref{NDproof} takes place in several steps. First, we reformulate the radial form of \eqref{Eq0} as a first order autonomous differential system using a new transformation, see \eqref{var3}.

\subsection{Formulation of our problem as a dynamical system} 
We first assume that $u$ is a positive radial (ND) solution of \eqref{Eq0}. 
We define
\begin{equation} \label{newb}
\vartheta:=\frac{s}{q-\crits+1},\quad 
\beta:=  \frac{\left(q-1\right)\vartheta}{2}-1,\quad \zeta:=\frac{2^\star(s)-2}{q-2^\star(s)+1}.
\end{equation}	 
We introduce a new transformation involving three functions $X_1$, $X_2$ and $X_3$ as follows 
\begin{equation} \label{var3}	
	X_1(t)=t\left( 1-\mu r^{s} u^{q-2^\star(s)+1}\right),\ \ 
	\quad 
	X_2(t)=\frac{1}{t},\quad X_3(t)=\frac{ru'(r)}{u(r)}+\vartheta, 
	\end{equation}
	where $t:=r^{-\beta}$ and $\beta,\vartheta$ are given by \eqref{newb}. Since $u$ is a positive radial (ND) solution of \eqref{Eq0}, that is, $ \lim_{r\to 0^+} r^{\vartheta} u(r)=\mu^{-1/(q-\crits+1)}$, it follows that
\begin{equation} \label{sens}\left\{ \begin{array}{l} 
1-X_1(t) X_2(t)=\mu r^s u(r)^{q-2^\star(s)+1}>0\quad \text{for all }t\in [2R^{-\beta},\infty),\\
 X_1(t) X_2(t)\to 0\ \text{as } t\to \infty.
\end{array} \right.
\end{equation} 
If we set $\vec X=(X_1,X_2,X_3)$, then, as one easily checks, we have that
\begin{equation} \label{grid} \vec X'(t) =(H_1(\vec X(t)), H_2(\vec X(t)),H_3(\vec X(t)))\end{equation}
for all  $t\in [2R^{-\beta},\infty)$, where $H_1$, $H_2$ and $H_3$ are real-valued functions defined on $\R^3$ by
\begin{equation} \label{ahh} \left\{ \begin{aligned} 
&H_1(\xi_1,\xi_2,\xi_3):= \xi_1 \xi_2+\beta^{-1} (q-2^\star(s)+1) (1-\xi_1 \xi_2) \xi_3,\\
& H_2(\xi_1,\xi_2,\xi_3):= -\xi_2^2,\\
& H_3(\xi_1,\xi_2,\xi_3):=\beta^{-1} \mu^{-\zeta} \xi_1 (1-\xi_1\xi_2)_+^{\zeta}+\beta^{-1} \xi_2 (\xi_3-\vartheta) (\xi_3-\vartheta+n-2).
\end{aligned} \right.
\end{equation}
By $\xi_+$ we mean the positive part of $\xi$. We define $\vec Y:=(Y_1, Y_2,Y_3)$, where $\vec Y(t)=\vec X(t+2R^{-\beta})$  
for all $ t\geq 0$. Then, \eqref{grid} gives that  
$\vec Y'(t)=(H_1(\vec Y(t)),H_2(\vec Y(t)), H_{3}(\vec Y(t)))$ for all $ t\in [0,\infty)$.
To get more regularity, for any $\varepsilon\in (0,1)$, we choose $\Psi_\varepsilon\in C^1(\R)$ such that 
$\Psi_\varepsilon(t)=t^{\zeta}$ for all $t\geq \varepsilon$. 
By choosing $\varepsilon_0\in (0,1)$ small enough and using \eqref{sens}, we find that 
	\begin{equation} \label{sys}
\vec Y'(t)=(H_1(\vec Y(t)),H_2(\vec Y(t)), H_{3,\Psi_\varepsilon}(\vec Y(t)))\quad \text{for all } t\in [0,\infty)
	\end{equation}   
for every $\varepsilon\in (0,\varepsilon_0)$, where the function $H_{3,\Psi_\varepsilon}:\R^3\to \R$ is defined by 
$$ H_{3,\Psi_\varepsilon}(\xi_1,\xi_2,\xi_3):= \beta^{-1} \mu^{-\zeta} \xi_1 \Psi_\varepsilon(1-\xi_1\xi_2)+\beta^{-1} \xi_2 (\xi_3-\vartheta) (\xi_3-\vartheta+n-2).
$$

\subsection{Existence of solutions for \eqref{sys}}
Using $\vartheta$, $\beta$ and $\zeta$ in \eqref{newb}, we define $\Upsilon$ and $\Gamma$ by 
\begin{equation} \label{upga} \Upsilon:=\mu^{\zeta/2} \sqrt{q-2^\star(s)+1}\quad \text{and}\quad \Gamma:=\vartheta\left(n-2-\vartheta \right)\mu^{\zeta}.
\end{equation}

 \begin{lemma} \label{exist:sol} Let $q\in (2^\star(s)-1,2^\star-1)$ and $\varepsilon\in (0,1)$. 
 Fix $\Psi_\varepsilon\in C^1(\R)$ such that $\Psi_\varepsilon(t)=t^{\zeta}$ for all $t\geq \varepsilon$. For every $\delta>0$ small, 
 there exist $r_0\in (0,\delta/2)$ and a Lipschitz function
 $w:[0,r_0]\times [-r_0,r_0]\to [-r_0,r_0]$ such that for any $(Y_{2,0},Z_{3,0})\in (0,r_0]\times [-r_0,r_0]$, the system \eqref{sys} subject to the initial condition
 \begin{equation} \label{iniy} \vec Y(0)=(\Upsilon (Z_{3,0}-w(Y_{2,0},Z_{3,0}))+\Gamma Y_{2,0} ,Y_{2,0}, w(Y_{2,0},Z_{3,0})+Z_{3,0})
\end{equation} has a solution $\vec Y(t)=(Y_1(t),Y_2(t),Y_3(t))$ for all $t\geq 0$ satisfying 
 \begin{equation} \label{aay} \lim_{t\to +\infty}  \vec Y(t)=(0,0,0).
 \end{equation}
 Moreover, we have $Y_2(t)=1/(t+Y_{2,0}^{-1})$ for all $t\geq 0$.  
 \end{lemma}
 
\begin{proof} 
Since $\Psi_\varepsilon(1)=1$, we find one critical point $(0,0,0)$ for \eqref{sys}. 
Linearising the flow around $(0,0,0)$, we get one {\em unstable} eigenvalue $\lambda_1=\mu^{-\zeta/2} \beta^{-1}\sqrt{q-2^\star(s)+1}$
with associated eigenvector $(\Upsilon,0,1)$, one {\em null} eigenvalue with associated eigenvector 
$(\Gamma,1,0)$ and one {\em stable} eigenvalue $-\lambda_1$ with associated eigenvector $(\Upsilon,0,1)$. 
For $\vec Z=(Z_1,Z_2,Z_3)$, using a  
change of coordinates 
\begin{equation} \label{vss} 
\vec Y=(\Upsilon (Z_1-Z_3)+\Gamma Z_2,Z_2,Z_1+Z_3), \ \text{i.e.,} \ \vec Z=\left(\frac{Y_1-\Gamma Y_2+\Upsilon Y_3}{2\Upsilon} ,Y_2, \frac{\Gamma Y_2+\Upsilon Y_3-Y_1}{2\Upsilon}\right),
\end{equation}
we bring the system \eqref{sys} to a diagonal form, namely
\begin{equation} \label{zig} \vec Z'(t)=(\lambda_1 Z_1(t)+h_1(\vec Z(t)),-Z_2^2(t), -\lambda_1 Z_3(t)+h_3(\vec Z(t)))\quad \text{for all } t\geq 0.
\end{equation} 
For any $\delta>0$ small, the functions $h_1$ and $h_3$ are $C^1$ on the ball $B_\delta(0)$ in $\mathbb R^3$ centred at $0$ with radius $\delta$. Moreover, for some constant $C_1>0$, the functions $h_1$ and $h_3$ satisfy
\begin{equation} \label{nano} 
|h_1(\vec \xi)|+ |h_3(\vec \xi)|\leq C_1\sum_{j=1}^3 \xi_j^2\  \text{and} \ 
|\nabla h_1(\vec \xi)|+|\nabla h_3(\vec \xi)|\leq C_1\sum_{j=1}^3 |\xi_j|
\end{equation}  
for all $ \vec \xi=(\xi_1,\xi_2,\xi_3)\in B_\delta(0)$. 
By \eqref{vss}, 
proving Lemma~\ref{exist:sol} is equivalent to showing that for every small $\delta>0$, 
there exist $r_0\in (0,\delta/2)$ and a Lipschitz map 
$w:[0,r_0]\times [-r_0,r_0]\to [-r_0,r_0]$ such that for all $(Y_{2,0},Z_{3,0})\in (0,r_0]\times [-r_0,r_0]$, the system \eqref{zig} subject to \begin{equation} \label{init} \vec Z(0)=(w(Y_{2,0},Z_{3,0}), Y_{2,0},Z_{3,0}) 
\end{equation} has a solution $\vec Z(t)$ for all $t\geq 0$ with 
$\lim_{t\to +\infty} \vec Z(t)=(0,0,0)$.  
Linearising the flow for \eqref{zig} around $(0,0,0)$ yields one null eigenvalue, and the classical Hartman--Grobman theorem does not apply to \eqref{zig}. In Appendix, using the notion of center-stable manifold and inspired by Kelley \cite{Kel}, we prove Theorem~\ref{71} that 
can be applied to \eqref{zig}
due to \eqref{nano}. This ends
the proof.  \qed
\end{proof}

\subsection{Proof of Proposition~\ref{NDproof}} For fixed $\varepsilon\in (0,1)$, we
choose $\Psi_\varepsilon\in C^1(\R)$ such that $\Psi_\varepsilon(t)=t^{\zeta}$ for all $t\geq \varepsilon$. Let $\delta\in (0,(1-\varepsilon)^{1/2})$. 
Let $r_0\in (0,\delta/2)$ and $w:[0,r_0]\times [-r_0,r_0]\to [-r_0,r_0]$ be given by Lemma~\ref{exist:sol}. We fix $Y_{2,0}:=r_0/2$. Then for any fixed $Z_{3,0}\in [-r_0,r_0]$, the system \eqref{sys}, subject to 
 the initial condition \eqref{iniy} has a solution $\vec Y(t)$ for all $t\geq 0$ such that 
\eqref{aay} holds. Moreover, we find that $Y_2(t)=1/(t+Y_{2,0}^{-1})$ for all $t\geq 0$. 
Let $t_0>0$ be large such that 
	$\vec Y(t)\in B_\delta (0)$ for all $ t\geq  t_0$.
	Using that  $0<\varepsilon< 1-\delta^2$,  for all $ t\geq t_0$, we get that $1-Y_1(t)Y_2(t)>\varepsilon$ so that
$ \Psi_\varepsilon(1-Y_1(t)Y_2(t))=(1-Y_1(t) Y_2(t))^\zeta$. 	
Hence, we have $H_{3,\Psi_\varepsilon}(\vec Y(t))=H_3(\vec Y(t)) $ for all $t\geq t_0$. 	
	For every $t\geq T:=t_0+Y_{2,0}^{-1}$, we define $\vec X(t)$ by 
	$ \vec X(t):=\vec Y(t-Y_{2,0}^{-1})
	$, which yields that $X_2(t)=1/t$.  Then, $\vec X(t)$ is a solution of the system \eqref{grid} for all $t\geq T $ such that 
	$ \lim_{t\to \infty} \vec X(t)=(0,0,0)$. With $\vartheta$ and $\beta$ be given by \eqref{newb} and $t:=r^{-\beta}$, we define $u(r)$ as in \eqref{var3}. Then $u$ is a positive radial (ND) solution of \eqref{Eq0} with  $R:=T^{-1/\beta}$. The above construction leads to an infinite number of positive radial (ND) solutions for \eqref{Eq0} by varying $Z_{3,0}$ in $[-r_0,r_0]$. This completes the proof. 
	\qed

\section{Consequences of Pohozaev's identity} \label{sec-thm2}	

In this section, using Pohozaev's identity, we prove Theorem~\ref{thm2}, followed by uniform {\em a priori} estimates for the positive radial solutions of \eqref{Eq0} satisfying \eqref{upli} (see Proposition~\ref{Pr}). 

Let $u$ be any positive solution of \eqref{Eq0} with $q\in (1,\crit-1)$ such that \eqref{upli} holds. As in
\cite{CirRob}, for every $r\in (0,R)$, we denote by $P_r^{(q)}(u) $ the Pohazev-type integral associated to $u$, namely 
\begin{equation} \label{poh-int} P_{r}^{(q)}(u):=\int_{\partial B(0,r)} \left[ (x,\nu) \left(\frac{|\nabla u|^2}{2}-\frac{u^{\crits}}{\crits|x|^s}+\mu\frac{u^{q+1}}{q+1}
\right)-T(x,u)\, \partial_\nu u\right]d\sigma,
\end{equation} where $T(x,u)=(x,\nabla u(x))+(n-2) u(x)/2$. Here, $\nu$ denotes the unit outward normal at $\partial B(0,r)$.  
Assuming $u$ satisfies \eqref{upli}, it was shown in \cite{CirRob} that there exists $\lim_{r\to 0^+}P_r^{(q)}(u)  :=P^{(q)}(u)$ and 
	\begin{equation} \label{van} P^{(q)}(u)\geq 0\end{equation} 
with strict inequality if and only if $u$ is a (CGS) solution of \eqref{Eq0}. We refer to $P^{(q)}(u)$ as the {\em asymptotic Pohozaev integral}.   
We introduce the notation 
\begin{equation} \label{lamb}  \lambda:=(n-2)(2^\star-1-q)/2\ \ \text{and}\ \ 
c_{\mu,q,n}:= \lambda\mu /(q+1).
\end{equation}
Both $\lambda$ and $c_{\mu,q,n} $ are positive by the assumption $q\in (1,\crit-1)$. 

\subsection{Proof of Theorem~\ref{thm2}} Let $q\in (1, \crit-1)$. 
Suppose that \eqref{diri} admits a positive smooth solution $u$ satisfying \eqref{upli}.  
From $u=0$ on $\partial \Omega$, we have $\nabla u=\left(\partial_\nu u\right) \nu$ for $x\in \partial \Omega$, where $\nu$ denotes the unit outward normal at $\partial \Omega$. 
For every $r>0$ small, 
by applying the Pohozaev identity as in \cite{CirRob}*{Proposition~6.1} for $\omega=\omega_r=\Omega\setminus \overline{B(0,r)}$, we get that
\begin{equation} \label{find}  -\frac{1}{2}\int_{\partial \Omega}  (x,\nu) |\nabla u|^2\,d\sigma = 
P_r^{(q)}(u) +c_{\mu,q,n}  \int_{\omega_r} u^{q+1}\,dx. 
\end{equation} 
 By letting $r\to 0^+$ in \eqref{find} and using \eqref{van}, we arrive at 
 \begin{equation} \label{find2}  -\frac{1}{2}\int_{\partial \Omega}  (x,\nu) |\nabla u|^2\,d\sigma = 
P^{(q)}(u) +c_{\mu,q,n}
 \int_{\Omega} u^{q+1}\,dx\geq 0. 
\end{equation} 
Since $\Omega$ is star-shaped with respect to the origin, we have $(x,\nu)>0$ on $\partial \Omega$. Then, \eqref{find2} can only hold when $\nabla u\equiv 0$ on $\partial \Omega$ and $u\equiv 0$ in $\Omega$. Hence, 
\eqref{diri} has no positive smooth solutions satisfying \eqref{upli}. Using the comments before statement of Theorem~\ref{thm2}, we finish the proof. \qed

\subsection{Uniform {\em a priori} estimates} 
Let $q\in (1,\crit-1)$. 
For the positive  
radial solutions $u$ of \eqref{Eq0} satisfying \eqref{upli}, we derive uniform {\em a priori} estimates.  
These are crucial for proving the existence of (MB) solutions in Proposition~\ref{mb} and (CGS) solutions in Proposition~\ref{CGS}. We define 
\begin{equation}
\label{defyz}
\left\{ 
\begin{aligned}
 & {\bar R}(u):=\sup \{R>0:\ u\ \text{is a positive radial solution of \eqref{Eq0}}\}, \\
& \ z(r):=r^{\frac{n-2}{2}} u(r)\ \text{for } r\in (0,R),\ \  F_0(\xi):=\frac{(n-2)^2}{4}\xi^2-\frac{2}{2^\star(s)}\xi^{2^\star(s)}\ \
\text{for } \xi\geq 0. 
\end{aligned}\right.
\end{equation}
If $u$ has a removable singularity at $0$ or $u$ is a solution of (MB) type,  then $\liminf_{r\to 0^+}z\(r\)=0$. If $u$ is a (CGS) solution, then  from~\cite{CirRob}, we can derive that
\begin{equation}\label{P7}
0<\liminf_{r\to 0^+}z\(r\)\le
\left[(n-2)/2\right]^{2/(2^\star(s)-2)}:=M_0.
\end{equation}
For $R>0$, we also define 
\begin{equation} \label{psir} F_R(\xi):=\frac{(n-2)^2}{4}-\frac{2}{2^\star(s)} \xi^{2^\star(s)-2}+
\frac{2\mu R^{\lambda} \xi^{q-1}}{q+1}\ \ \text{for } \xi\geq 0.\end{equation}
For $F_0$ given by \eqref{defyz}, let $\Lambda_0$ denote the unique positive solution of $F_0(\xi)=0$, that is 
\begin{equation}\label{PrEq1}
\Lambda_0:=\[(n-2)(n-s)/4\]^{\frac{1}{2^\star(s)-2}}. 
\end{equation} 
For any $\Lambda>\Lambda_0$, we have $F_0(\Lambda)<0$. Let $R_\Lambda$ denote the unique $R>0$ for which $F_{R}(\Lambda)=0$: 
\begin{equation} \label{rlambda} R_\Lambda:= \left[-\frac{(q+1)F_0(\Lambda)}{2\mu \Lambda^{q+1}}\right]^{\frac{1}{\lambda}}>0.
\end{equation}
	Moreover, it holds
\begin{equation}\label{RLambda}
R_\Lambda=\sup\left\{\xi>0:\,F_{R_\Lambda}\(t\)>0\quad \text{for all } t\in (0,\xi)\right\}.
\end{equation}

\begin{proposition}[Uniform {\em a priori} estimates] \label{Pr} Let $q\in (1,2^\star-1)$. Then for every $\Lambda>\Lambda_0$, there exists $R_\Lambda>0$ as in \eqref{rlambda} such that any positive radial solution of \eqref{Eq0} with $R\in (0,R_\Lambda)$ satisfying \eqref{upli} can be extended as a positive radial 
	solution of \eqref{Eq0} in $B(0,R_\Lambda]\setminus\{0\}$  and
	\begin{equation}\label{PrEq2}
	r^{\frac{n-2}{2}}u\(r\)<\Lambda \qquad\text{for all } r\in (0,R_\Lambda].
	\end{equation}
\end{proposition}
Let $\omega_{n-1}$ denote the volume of the Euclidean 
$(n-1)$-sphere $\mathbb S^{n-1}$ in $\R^n$. Let $\lambda$ and $c_{\mu,q,n}$ be given by \eqref{lamb}. For $q>\crits-1$, we define $\ell_q$ as follows
	\begin{equation} \label{grand}	\ell_q:=
\frac{(2-s)(q+1)}{ (n-s)(q-1)} \left[\frac{(n-2)(n-s)(q-1)}{4(q-2^\star(s)+1)}\right]^{-\frac{q-2^\star(s)+1}{2^\star(s)-2}}. \end{equation}

\noindent A key tool in proving Proposition~\ref{Pr} is given by Lemma~\ref{pos}, which is of interest in its own.  

\begin{lemma} \label{pos}  Let $q\in (1,\crit-1)$. Let $u$ be a positive radial solution of \eqref{Eq0} satisfying
	\eqref{upli}. 
	\begin{itemize}
		\item[{\rm (a)}] For all $r\in (0,\bar R)$, the functions $z$ and $F_r(z)$ in \eqref{defyz} and \eqref{psir}, respectively satisfy 
		\begin{equation} \label{hihi1} z^2(r) \,F_r(z(r))
		=\frac{2P^{(q)}(u)}{\omega_{n-1}}+ [rz'(r)]^2+2 c_{\mu,q,n} \int_0^r \xi^{n-1} u^{q+1}(\xi)\,d\xi.
		\end{equation} 
		\item[{\rm (b)}] If $\bar R<+\infty$, then  
		$\liminf_{r\nearrow \bar R} u(r)>0$ and $\limsup_{r\nearrow \bar R}u(r)=+\infty$. 		
		\item[{\rm (c)}] If $1<q<\crits-1$, then $\bar R=+\infty$. 
		\item[{\rm (d)}] If $q=\crits-1$, then $\bar R\geq (1/\mu)^{1/s}$. 
		\item[{\rm (e)}] If $q\in (\crits-1,\crit-1)$, then $\bar R>(\ell_q/\mu)^{1/\lambda}$, where 
			$\ell_q$ is given by \eqref{grand}. 		
		\end{itemize} 
\end{lemma}

\begin{remark} We have $\ell_q\to 1$ as $q\searrow 2^\star(s)-1$ and using $F_0$ in \eqref{defyz}, we get
	\begin{equation} \label{lqsup}
	\ell_q=\frac{q+1}{2}\sup_{\Lambda\in (\Lambda_0,\infty)} \frac{-F_0(\Lambda)}{\Lambda^{q+1}}.
	\end{equation}
\end{remark}

\begin{proof} From our assumptions, it follows that $\lim_{r\to 0^+} r^n u^{q+1}(r)=0$. 
	\begin{proof}[Proof of (a)]
			Since $u$ is a radial solution of \eqref{Eq0}, 
	the {\em Pohozaev-type 
			integral} $P_r^{(q)}(u)$ satisfies 
		\begin{equation} \label{red} \frac{2P_{r}^{(q)}(u)}{\omega_{n-1}} = 
		-[rz'(r)]^2+z^2(r)\,F_r(z(r))\quad \text{for all } r\in (0,\bar R). 
		\end{equation}
		By the Pohozaev identity, see \cite{CirRob}*{Proposition~6.1}, for every $0<r_1<r<\bar R$, we find that 
		\begin{equation} \label{dif}
		P_{r}^{(q)}(u)-P_{r_1}^{(q)}(u)=\omega_{n-1} c_{\mu,q,n} \int_{r_1}^{r} \xi^{n-1} u^{q+1}(\xi)\,d\xi.
		\end{equation} 
	Letting $r_1\to 0^+$ in \eqref{dif}, for any $r\in (0,\bar R)$, we find that 
		\begin{equation} \label{trei}
		P_r^{(q)}(u)=P^{(q)}(u)+\omega_{n-1} c_{\mu,q,n} \int_0^r \xi^{n-1} u^{q+1}(\xi)\,d\xi. 
		\end{equation}
	Then we conclude \eqref{hihi1} by using \eqref{red} and \eqref{trei}. The proof of (a) is now complete. \qed
	\end{proof}

	\begin{proof}[Proof of (b)] 
		Assume that $\bar R<+\infty$. To prove that $\liminf_{r\nearrow \bar R}u(r)>0$, we proceed by contradiction. Assume that for a sequence $(r_k)_{k\geq 1}$ of positive numbers with $r_k\nearrow \bar R$ as $k\to \infty$, we have 
		$\lim_{k\to \infty} u(r_k)= 0$, that is $\lim_{k\to \infty} z(r_k)=0$. We let $r=r_k$ in \eqref{hihi1}, then pass to the limit $k\to \infty$ to obtain 
		a contradiction. 
		   	For the other claim in (b), assume that $\limsup_{r\nearrow \bar R} u(r)<+\infty$. Then $\limsup_{r\nearrow \bar R} z(r)<+\infty$ since $\bar R<+\infty$. By the classical ODE theory, it follows that $\limsup_{r\nearrow \bar R}|u'(r)|=\infty$.
			On the other hand, by \eqref{hihi1}, we get that $\limsup_{r\nearrow \bar R} |rz'(r)|<+\infty$, which   
		shows that 
		$\limsup_{r\nearrow \bar R} |u'(r)|<+\infty$.  
		This contradiction completes the proof of (b).	  		\qed  
	\end{proof}

\begin{proof}[Proof of (c)] Let $q<2^\star(s)-1$. 
	If $\bar R<\infty$, then there exists a sequence $(r_k)_{k\geq 1}$ in $(0,\bar R)$ with $\lim_{k\to \infty}r_k=\bar R$ and $\lim_{k\to \infty} z(r_k)= +\infty$. By letting $r=r_k$ in \eqref{hihi1} and $k\to \infty$, the left-hand side of \eqref{hihi1} diverges to $-\infty$ as $k\to \infty$, which is a contradiction. This proves that $\bar R=+\infty$. \qed \end{proof}

\begin{proof}[Proof of (d)] 
	Let $q=2^\star(s)-1$. We argue by contradiction. Assume that $\bar R<\left(1/\mu\right)^{1/s}$. Then,  
	there exists $(r_k)_{k\geq 1}$ in $(0,\bar R)$ with $\lim_{k\to \infty}r_k= \bar R$ and $\lim_{k\to \infty}z(r_k)= +\infty$.  
	Since $r^{\lambda}=r^{s}<\bar R^{s}$ for all $r\in (0,\bar R)$, from \eqref{hihi1} and the definition of $F_r$ in \eqref{psir} (with $R=r$), we have  
	\begin{equation} \label{figu}  \frac{(n-2)^2 z^2(r_k)}{4}-
	\frac{2\left(1-\mu \bar R^{s}\right)z^{2^\star(s)}(r_k)}{2^\star(s)}> 0 \quad \text{for all }
	k\geq 1.
	\end{equation}   
	By letting $k\to \infty$ in \eqref{figu} and using that $1-\mu \bar R^{s}>0$, we get that the left-hand side of \eqref{figu} tends to $-\infty$ as $k\to \infty$. This contradiction proves that $\bar R\geq \left(1/\mu\right)^{1/s}$. \qed \end{proof}

\begin{proof}[Proof of (e)]  
	Let $q\in (2^\star(s)-1,2^\star-1)$. To prove $\bar R>\left(\ell_q/\mu\right)^{1/\lambda}$ with
	$\ell_q$ as in \eqref{grand},  
	it suffices to assume $\bar R<+\infty$. 
	Let $F_{\bar{R}}$ be the function $F_R$ in \eqref{psir} with $R=\bar R$. 
	We distinguish two cases:

	\smallskip\noindent {\sc Case 1}: If $u$ has a removable singularity at $0$, or $u$ is a (MB) solution, then 
	$\liminf_{r\to 0^+} z(r)=0$ using that $z(r)=r^{\frac{n-2}{2}} u(r)$. 
	Since $\limsup_{r\nearrow {\bar R}} z(r)=+\infty$, to ensure \eqref{hihi1} for a positive radial solution $u$ of \eqref{Eq0} which is {\em not} (CGS)  {\em nor} (ND), it is necessary to have 
	\begin{equation} \label{exi}
	F_{\bar{R}}(\xi)>0 \quad \text{for all }
	\xi\in [0,\infty).  	
	\end{equation} 
	We next study the monotonicity of $F_{\bar{R}}$. We see that $F_{\bar{R}}$ has only one positive critical point $\xi_c$ defined by
	\begin{equation} \label{definX} \xi_c:=\left( \frac{(2-s)(q+1)}{\mu (n-s)(q-1) \bar R^{\lambda} }\right)^{\frac{1}{q-2^\star(s)+1}}. 
	\end{equation} Moreover, $\xi_c$ is a global minimum point for $\bar F$ on $[0,\infty)$. Thus, \eqref{exi} holds if and only if 
	$F_{\bar{R}}(\xi_c)> 0$, which corresponds to $\bar R>\left(\ell_q/\mu\right)^{1/\lambda}$.

	\vspace{0.2cm}	
	\noindent {\sc Case 2:} If $u$ is a radial (CGS) solution of \eqref{Eq0} then we need $F_{\bar{R}}(\xi)>0$ for every $\xi\geq \liminf_{r\to 0^+} z(r)$. If $M_0$ in \eqref{P7} satisfies $M_0\leq \xi_c$ then $\bar R>\left(\ell_q/\mu\right)^{1/\lambda}$ is necessary to have $F_{\bar{R}}(\xi)>0$ for every $\xi\in [\liminf_{r\to 0^+} z(r),+\infty)$.     
	If $M_0> \xi_c $, then from \eqref{definX} and \eqref{P7}, we get
	\begin{equation} \label{minim2} 
	\bar R^\lambda > 
	\frac{(2-s)(q+1)}{\mu (n-s)(q-1)} \left(\frac{n-2}{2}\right)^{-\frac{2(q-2^\star(s)+1)}{2^\star(s)-2}},
	\end{equation}	which again implies $\bar R>\left(\ell_q/\mu\right)^{1/\lambda}$.
	
	\vspace{0.2cm}
	\noindent We have established the assertion of (e) in both Cases 1 and 2. \qed
\end{proof} 
\noindent This completes the proof of Lemma~\ref{pos}.  \qed
	\end{proof}



%





\noindent {\bf Proof of Proposition~\ref{Pr}.} 
	For any $q\in [2^\star(s)-1,2^\star-1)$, we 
	denote
	$R^*=R^*(q)$ as follows
	$$ R^*:=\left\{\begin{aligned}
	& (1/\mu)^{1/s} && \text{if } q=2^\star(s)-1,&\\
	& (\ell_q/\mu)^{1/\lambda} && \text{if } 2^\star(s)-1<q<2^\star-1.&
	\end{aligned} \right.$$		
Let $\Lambda>\Lambda_0$ be fixed. Let $u$ be any positive radial solution of \eqref{Eq0} with $R\in (0,R_\Lambda)$ such that \eqref{upli} holds. From Lemma~\ref{pos}, 
	the maximum radius of existence $\bar R=\bar R(u)$ for $u$ satisfies
	$\bar R=+\infty$ if $1<q<2^\star(s)-1$,  $\bar R\geq R^*$ for $q=2^\star(s)-1$ and $\bar R>R^*$ for $2^\star(s)-1<q<2^\star-1$.
	From \eqref{lqsup} and \eqref{rlambda}, we have $R_\Lambda\leq R^*$ for all $2^\star(s)-1< q<2^\star-1$. When $q=2^\star(s)-1$, then using the definition of $F_0$ and $R^*$, we see easily that $R_\Lambda<R^*$.  
Hence, we can extend $u$ as a positive radial solution of \eqref{Eq0} in $B(0,R_\Lambda]\setminus\{0\}$ for all $1<q<2^\star-1$.  

\smallskip\noindent 
We now prove \eqref{PrEq2}.  
Assume by contradiction that \eqref{PrEq2} fails, that is, $z(r_0)\geq \Lambda$
for some $r_0\in (0,R_\Lambda]$, where $z(r):=r^{\frac{n-2}{2}}u(r)$ is defined as in \eqref{defyz}.  
Since $z(r_0)\geq \Lambda>\Lambda_0>M_0$, the Mean Value Theorem, together with \eqref{RLambda} and \eqref{P7}, gives that there exists $r_1\in\(0,r_0\)$ such that $z\(r_1\)=\Lambda$. Hence, using Lemma~\ref{pos}(a), we find that 
$ 0= \Lambda^2\,F_{R_{\Lambda}}(\Lambda) >0$.  This contradiction ends the proof of Proposition~\ref{Pr}. \qed




\section{Removable singularities}\label{Sec3}  

The assertion of Theorem~\ref{Th1}(i) follows from Lemma~\ref{pos} and Lemma~\ref{remove} below.     

\begin{lemma} \label{remove} For $q>1$ and every
	$\gamma\in (0,\infty)$, there exists $R>0$
	such that \eqref{Eq0} has a unique positive radial solution $u_\gamma$ 
	with a removable singularity at $0$ and 
	$\lim_{r\to 0^+} u_\gamma(r)=\gamma$.
\end{lemma}

\begin{proof} Fix $\gamma\in (0,\infty)$  arbitrarily. 
We consider the following initial value problem:
	\begin{equation} \label{rem}
	\left\{ \begin{aligned}
	&  y''(\xi)+a\, y'(\xi)/\xi+4 (
		y^{2^\star(s)-1}-\mu\, \xi^{\frac{2s}{2-s}}\,y^q)/(2-s)^2 =0 \   \text{for  } \xi>0,\\
	&   y(0)=\gamma,\quad y'(0)=0,
	\end{aligned} \right.
	\end{equation}
	where we denote $a:=(2n-s-2)/(2-s)$. 
	By Biles--Robinson--Spraker \cite{BRS}*{Theorems~1 and 2}, 
	for every $\gamma>0$,  
	there exists a unique positive solution $y_\gamma$ 
	of \eqref{rem} on some interval $[0,T]$ 
	with $T>0$. A solution $y$ of \eqref{rem} is 
	defined in \cite{BRS} as follows:
	\begin{enumerate}
		\item[(a)] $y$ and $y'$ are absolutely continuous on $[0,T]$;
		\item[(b)] $y$ satisfies the ODE in \eqref{rem} a.e. on $[0,T]$;
		\item[(c)] $y$ satisfies the initial conditions in \eqref{rem}. 
	\end{enumerate}
Since $a>1$, the function $\xi\longmapsto \xi^a y_\gamma '(\xi)$ is absolutely continuous  
	on $[0,T]$. From \eqref{rem}, we have 
	$$ \left(\xi^a y_\gamma '(\xi)\right)'= -\frac{4}{(2-s)^2} \xi^a \left( 
	y_\gamma ^{2^\star(s)-1}-\mu\, \xi^{\frac{2s}{2-s}}\,y_\gamma ^q
	\right)\quad \text{a.e. in } [0,T].
	$$ 
	Thus, for all $\xi\in [0,T]$, we find that
	$$  \xi^a y_\gamma '(\xi)=-\frac{4}{(2-s)^2} \int_0^\xi 
	t^a \left( 
	y_\gamma ^{2^\star(s)-1}(t)-\mu\, t^{\frac{2s}{2-s}}\,y_\gamma ^q(t)\right)\,dt.
	$$
By the property (a) for $y_\gamma$, we  
	find that $y_\gamma \in C^2(0,T]$ satisfies the ODE in \eqref{rem}
	on $ (0,T]$. 
	The change of variable $u_\gamma(r)=y_\gamma(\xi)$ with $\xi=r^{(2-s)/2}$ yields that 
	$u_\gamma$ is a positive radial $C^2(0,R]$-solution 
	of \eqref{Eq0} with $R=T^{2/(2-s)}$ and 
	$\lim_{r\to 0^+}u_\gamma(r)=\gamma$. This proves the existence claim.
	
	\vspace{0.2cm}\noindent We now show the uniqueness claim: any 
	positive radial $C^2(0,R]$-solution $u$ of \eqref{Eq0} for some $R>0$ 
	such that $\lim_{r\to 0^+}u(r)=\gamma$ must coincide with 
	$u_\gamma$ on their common domain of existence. Indeed, using the
	change of variable $u(r)=y(\xi)$ with $\xi=r^{(2-s)/2}$, we get that 
	$y\in C^2(0,R^{(2-s)/2}]$ satisfies the differential equation in \eqref{rem} for all 
	$\xi\in (0,R^{(2-s)/2})$ and 
	$\lim_{\xi \to 0^+}y(\xi)=\lim_{r\to 0^+}u(r)=\gamma$. Hence, $y$ can be 
	extended by continuity at $0$ by defining $y(0)=\gamma$. 
	To conclude 
	that $y$ is a solution of \eqref{rem} on $[0,R^{(2-s)/2}]$ 
	in the sense of \cite{BRS}, that is, $y$ 
	satisfies properties (a)--(c) stated above with $T=R^{(2-s)/2}$, it suffices  
	to show that  
	\begin{equation} \label{f} 
	y'(\xi)\to 0\ \text{and }\ 
	y''(\xi)\to -2\,\gamma^{2^\star(s)-1}/[(n-s)(2-s)]\ \ \text{as }  \xi\to 0^+.
	\end{equation} 
	This would give that $y\in C^2[0,R^{(2-s)/2}]$, and then, by applying Theorem~2 in 
	\cite{BRS}, we conclude that $y=y_\gamma$ on $[0,\min\{T, R^{(2-s)/2}\}]$, 
	proving our uniqueness assertion.

\medskip \noindent 
We prove \eqref{f}. Since $u$ is a positive radial solution of \eqref{Eq0} with $\lim_{r\to 0^+} u(r)=\gamma$, we have  
	\begin{equation} \label{rad}
	r^{-(n-1-s)} \left( r^{n-1} u'(r) \right)'=-u^{\crits-1}+\mu r^s u^q\to 
	-\gamma^{\crits-1} \quad \text{as } r\to 0^+.
	\end{equation}
	Hence, the function 
	$r\longmapsto r^{n-1} u'(r)$ is decreasing on some interval $(0,r_0)$ 
	for small 
	$r_0>0$. Thus, there exists $\lim_{r\to 0^+} r^{n-1} u'(r)=
	\theta\in (-\infty, \infty]$. We next show that $\theta=0$.
	Assume by contradiction that $\theta\not= 0$. Then choosing 
	$\min\{\theta,0\}<c<\max\{\theta,0\}$, 
	we find that $h(r)= u(r)+c(n-2)^{-1} r^{2-n}$ 
	is decreasing (respectively, increasing) on $(0,r_1)$ for $r_1>0$ 
	small when $\theta<0$ 
	(respectively, when $\theta>0$). Since $\lim_{r\to 0^+} h(r)=-\infty$ 
	if $\theta<0$ and $\lim_{r\to 0^+} h(r)=+\infty$ if $\theta>0$, 
	we arrive at a contradiction. This proves that 
	$\lim_{r\to 0^+} r^{n-1} u'(r)=0.$ Hence, by \eqref{rad}, we get that 
$	\lim_{r\to 0^+}r^{s-1}u'(r)=
	-\gamma^{\crits-1}/(n-s)$. 
Coming back to the $\xi$ variable, we obtain \eqref{f}. 
This ends the proof of Lemma~\ref{remove}.    \qed
\end{proof}

\section{(MB) solutions}\label{Sec4}

In Sect.~\ref{sec-51} we prove Corollary~\ref{corol}. In Sect.~\ref{sec-52} we prove Theorem~\ref{Th1}(ii) given as Proposition~\ref{mb}. 

\subsection{Proof of Corollary~\ref{corol}} \label{sec-51}
For every $0<\ell<\min\{(n-2)/2,2\}$, 
we set $q:=\crit -1 -2\ell/(n-2)$ so that 
$q\in  (\crit-2,\crit-1)$ with $q>1$. Then, for every $s\in (0,2)$, Theorem~\ref{Th1}(ii) yields a positive radial (MB) solution $u_{MB}$ of \eqref{Eq0} for some $R>0$.  
We define $z(r)=r^{(n-2)/2} u_{MB}(r)$ for $r\in (0,R)$. Since  
$z^*:=\limsup_{r\to 0^+}z(r)\in (0,\infty)$ and 
$z_*=\liminf_{r\to 0^+} z(r)=0$, the asymptotics of $u_{MB}$ at zero is different from that of any positive singular solution of \eqref{rnsol}. 
By defining 
$$K(r)=1-\mu r^s u_{MB}(r)^{q-\crits+1}\ \text{and } C_{s,\ell}:=2(s-\ell)/(n-2)\  \text{for } r= |x|\in (0,R),$$
we see that $u=u_{MB} $ is a positive singular solution of \eqref{sgt}. Moreover, we find that
\begin{equation} \label{home}
|r^{1-\ell} K'(r)| =\mu [z(r)]^{C_{s,\ell}}\left|\ell+ C_{s,\ell}\,rz'(r)/z(r)\right|\quad \text{for all } r\in (0,R) . \end{equation}
We have $C_{s,\ell}>0$ when $\ell<s$ and $C_{s,\ell}<0$ when $\ell>s$. With $\underline L$ and $\overline L$ as in \eqref{gag}, we prove that 
\begin{equation} \label{scor} \underline{L}=0<\overline L<\infty\ \text{if } 
\ell\in (0, s),\ \text{ whereas } 0<\underline L<\overline L=\infty\ 
\text{ if } \ell>s.\end{equation}  
Indeed, since $P^{(q)}(u_{MB})=0$ and $z^*<\infty$, Lemma~\ref{pos}(a) yields that  
\begin{equation} \label{vic} \limsup_{r\to 0^+} r|z'(r)|/z(r)<\infty \ \text{ and } F_0(z(r))-[rz'(r)]^2\to 0\ \text{ as }r\to 0^+,\end{equation} where $F_0$ is given by \eqref{defyz}. Hence, $\underline L=0$ and $\overline L<\infty$ if $\ell\in (0,s)$. Since $z_*=0$, for every $\rho\in (0,z^*)$, there exists a sequence $\{r_k\}$ of positive numbers decreasing to $0$ as $k\to \infty$ such that 
$\lim_{k\to \infty} z(r_k)=\rho$. Then, by \eqref{vic}, we have $\lim_{k\to \infty} (r_kz'(r_k))^2=F_0(\rho)$. For suitable $\rho$, using $r_k$ in \eqref{home}, we get that $\overline L>0$ for $\ell\in (0,s)$, and correspondingly $0<\underline L<\infty$ for $\ell>s$. 
It remains to show that $\overline L=\infty$ if $\ell>s$. Assuming the contrary, $R_k z'(R_k)/z(R_k)\to -\ell/C_{s,\ell}$ for every sequence $\{R_k\}$ of positive numbers decreasing to $0$ such that 
$\lim_{k\to \infty}z(R_k)= 0$. Lemma~\ref{pos}(a) gives that 
$ F_{R_k}(z(R_k))\geq [R_k z'(R_k)/z(R_k)]^2$. Letting $k\to \infty$,  we would have $(n-2)^2/4\geq \ell^2/C^2_{s,\ell}$, which is a contradiction with $s>0$. 
Thus, \eqref{scor} holds and $K$ satisfies the properties in Corollary~\ref{corol}. \qed

\subsection{Existence of (MB) solutions} \label{sec-52}

\begin{proposition} \label{mb} Let $q\in (1,2^\star-1)$. Assuming that $q> 2^\star-2$, then for $R>0$ small, \eqref{Eq0} admits at least a positive radial (MB) solution $u$, that is,  
	\begin{equation} \label{mbsol}
\liminf_{r\to 0^+} r^{\frac{n-2}{2}} u(r)=0\quad \text{and} \quad 
	\limsup_{r\to 0^+} r^{\frac{n-2}{2}} u(r)\in (0,\infty). 
	\end{equation}
\end{proposition}

\begin{proof}
	We use an argument inspired by Chen--Lin \cite{ChenLin}. 
	Let $(\gamma_i)_{i\geq 1}$ be an increasing sequence of positive numbers with $\lim_{i\to \infty} \gamma_i=\infty$. 
By Lemmas~\ref{remove} and \ref{pos}, for every 	
	$i\geq 1$, there exists $R_i>0$ such that \eqref{Eq0}, subject to $\lim_{|x|\to 0^+} u(x)=\gamma_i$, admits a unique positive radial $C^2(0,R_i]$-solution $u_{\gamma_i}$. From now on, we use $u_i$ instead of $u_{\gamma_i}$. Let $\Lambda>\Lambda_0$ be fixed, where $\Lambda_0$ is given by \eqref{PrEq1}. By Proposition~\ref{Pr}, there exists $R_\Lambda>0$ such that $u_i$ can be extended as a positive radial $C^2(0,R_\Lambda]\cap C[0,R_\Lambda]$-solution of \eqref{Eq0} in $(0,R_\Lambda]$ satisfying  \begin{equation} \label{upp}
	u_i(0)=\gamma_i,\quad
	r^{\frac{n-2}{2}} u_i(r)\leq \Lambda\ \text{for all } r\in (0,R_\Lambda]\ \text{and every } i\geq 1. 
	\end{equation}

\vspace{0.2cm}	
\noindent {\sc Claim: } For any $u_0>0$, there exist $r_0\in (0,R_\Lambda)$ and $i_0\geq 1$ such that 
	$$ u_i(r_0)\geq u_0\quad \text{for all } i\geq i_0.  
	$$    
	
	\noindent We now complete the proof of Proposition~\ref{mb} assuming the Claim. 
	From \eqref{upp}, there exists a subsequence of $(u_i)$, relabelled 
	$(u_i)$, converging uniformly to
	$u_{\infty}$ on any compact subset of 
	$(0,R_\Lambda]$. Moreover, $u_i\to u_\infty$ in $C^2_{\rm loc}(0,R_\Lambda]$ and $u_\infty$ is a radial solution of \eqref{Eq0}. The above Claim yields  $\limsup_{r\to 0^+} u_\infty(r)=\infty$, that is, $u_\infty$ has a non-removable singularity at $0$. By \eqref{upp}, we get   
	$\limsup_{r\to 0^+} r^{\frac{n-2}{2}}u_\infty(r)\in (0,\infty)$. 
	Since $q<2^\star-1$, we thus find that $u_\infty^{q+1}\in L^1(B(0,R_\Lambda))$.
	We have $P^{(q)}(u_i)=0$ for all $i\geq 1$. By letting $u=u_i$ in \eqref{trei} and \eqref{red}, then passing to the limit $i\to +\infty$, we find that 
	\begin{equation} \label{uui}  P_r^{(q)}(u_\infty)= c_{\mu,q,n} \int_{B(0,r)} u_\infty^{q+1}(x)\,dx
	\quad \text{for all } r\in (0,R_\Lambda]. 
	\end{equation}  By letting $r\to 0^+$ in \eqref{uui}, we find that $P^{(q)}(u_\infty)=0$. Hence by \eqref{van}, $u_\infty$ is not a (CGS) solution of \eqref{Eq0}. As $u_\infty$ does not have a removable singularity at $0$, we conclude that 
	$u_\infty$ is a radial (MB) solution of \eqref{Eq0}, that is $u_\infty$ satisfies \eqref{mbsol}. 
	This ends the proof of Proposition~\ref{mb}.	\qed \end{proof}
	
	\begin{proof}[Proof of the Claim] 
		Suppose the contrary. Then for some  $u_0>0$ and any $r_0\in (0, R_\Lambda)$, there exists a subsequence 
	of $u_i$, relabeled $(u_i)$, such that 
	\begin{equation} \label{mim} u_i(r_0)<u_0\quad \text{for all } i\geq 1. \end{equation}
	We apply the following transformation 	\begin{equation} \label{tr} w_i(t)=r^{\frac{n-2}{2}} u_i(r)\quad \text{with } t=\log r.\end{equation}
By $w_i'(t)$ and $w_i''(t)$, we denote the first and second derivative of $w_i$ with respect to $t$, respectively. 
Then $w_i$ satisfies the equation
	\begin{equation} \label{weq}
	w_i''(t)-f(w_i(t)) =\mu e^{\lambda t} w_i^q(t)\quad  
	\text{for } -\infty<t< \log R_\Lambda,
	\end{equation} 
	where $\lambda:=(n-2)(2^\star-1-q)/2$ and $f:[0,\infty)\to \R $ is defined by  
\begin{equation}\label{fdefi}
f(\xi):=(n-2)^2\xi/{4} -\xi^{2^\star(s)-1}\quad \text{for all } \xi\geq 0.
\end{equation}
	From \eqref{upp}, we have that 
	\begin{equation} \label{ww} w_i(t)\in (0,\Lambda]\quad \text{ for all } t\in (-\infty,\log R_\Lambda]\ \text{and } i\geq 1.\end{equation}

\noindent The proof of the Claim is now divided into five steps:

\begin{step} \label{step51}
The family  $(w_i'(t))_{i\geq 1}$ is uniformly bounded on $(-\infty, \log R_\Lambda]$. \end{step}

\begin{proof}[Proof of Step \ref{step51}] 
Using $F_R$ in \eqref{psir} with $R=e^t$, we define $E_i:(-\infty,\log R_\Lambda]\to \R$ by 
\begin{equation} \label{ei} E_i(t):=
\left(w_i'(t)\right)^2-w_i^2(t)\,F_{e^t} (w_i(t)). 
\end{equation} 
We have $\lambda>0$ (since $q<2^\star-1$) and  
$\lim_{t\to -\infty} w_i(t)=0
$. 
By Lemma~\ref{pos}(a), we find that 
\begin{equation}\label{eii}  E_i(t)=-2 c_{\mu,q,n}\int_{-\infty}^t 
e^{\lambda \xi} w_i^{q+1}(\xi)\,d\xi\ \text{ and }\ 
 E_i'(t)=-2 c_{\mu,q,n}\,e^{\lambda t} w_i^{q+1}(t)<0 
\end{equation} 
for all 
$t\in (-\infty,\log R_\Lambda)$. 
It follows that 
\begin{equation} \label{limei} \lim_{t\to -\infty} w_i'(t)=\lim_{t\to -\infty} E_i(t)=0.\end{equation}
From \eqref{eii}, we have $E_i<0$ on $(-\infty,\log R_\Lambda]$. Thus, by \eqref{ww}, we get that 
$(w_i'(t))_{i\geq 1}$ is uniformly bounded for $t\in (-\infty,\log R_\Lambda]$, completing Step~\ref{step51}.  \qed \end{proof}

\begin{step} \label{step52} For $\varepsilon_0>0$ and $r_0\in (0,R_\Lambda)$ small such that $ r_0^{(n-2)/2} u_0<\varepsilon_0/2$, we set
$$ \mathcal F_i:=\left\{ t\in (-\infty, \log r_0):\ w_i(t)\geq \varepsilon_0\right\}\quad \text{for all } i\geq 1.$$
Then there exists $i_0\geq 1$ such that 
$$w_i(\log r_0)<\varepsilon_0/2\hbox{ and }\mathcal F_i\not=\emptyset \quad \text{for every } i\geq i_0.$$
\end{step}

\begin{proof}[Proof of Step~\ref{step52}]
For $0<\varepsilon_0<\left[(n-2)/2\right]
^{(n-2)/(2-s)}$, we define 
\begin{equation} \label{beta} \beta_0:=2 \varepsilon_0^{2^\star(s)-2}\left(n-2+
	\sqrt{(n-2)^2-4\varepsilon_0^{2^\star(s)-2}}\right)^{-1}  \hbox{ so that }0<\beta_0<\frac{n-2}{2}\hbox{ is small}.
\end{equation}    
Since $\beta_0\to 0^+$ as $\varepsilon_0\to 0^+$, we can take $\varepsilon_0>0$  small enough such that $\beta_0$ is smaller than $\max\{(n-2)/4, 2/q, (2-s)/(2^\star(s)-1)\}$. Our choice of $r_0$ and \eqref{mim} yield that 
\begin{equation} \label{r0} w_i(\log r_0)=r_0^{\frac{n-2}{2}} u_i(r_0)<r_0^{\frac{n-2}{2}} u_0<\varepsilon_0/2\quad 
\text{for all } i\geq 1. 
\end{equation}
To end Step~\ref{step52}, we show by contradiction that there exists $i_0\geq 1$ such that $\mathcal F_i\not=\emptyset$ for all
$i\geq i_0$. Indeed, suppose that for a subsequence $(w_{i_k})_{k\geq 1}$ of $(w_i)_{i\geq 1}$, we have \begin{equation} \label{gr} w_{i_k}(t)<\varepsilon_0\quad \text{ for all } 
t\in (-\infty, \log r_0]\ \text{ and every } k\geq 1. \end{equation} 
Let $k\geq 1$ be arbitrary.  
Using \eqref{gr} and \eqref{beta}, we infer that 
\begin{equation} \label{es} \beta_0\left(n-2-\beta_0\right)=\varepsilon_0^{2^\star(s)-2} > w_{i_k}^{2^\star(s)-2}(t)
\end{equation} for all $t\leq \log r_0$.  
From \eqref{weq} and \eqref{es}, we obtain that 
\begin{equation} \label{sec} w_{i_k}''(t)> \left[(n-2)/2-\beta_0\right]^2 w_{i_k}(t)\quad \text{for all } t\leq \log r_0.  
\end{equation} In particular, $t\longmapsto w_{i_k}'(t)$ is increasing on $(-\infty,\log r_0]$. Since $\lim_{t\to -\infty} w_{i_k}'(t)=0$, we find that 
$w_{i_k}'(t)>0$ for all $t\leq \log r_0$. Set
$$\mathcal G_{i_k}(t) :=(w_{i_k}'(t))^2-
\left[(n-2)/2-\beta_0\right]^2 w_{i_k}^2(t).$$
Using \eqref{sec}, we get that
$\mathcal G_{i_k}$ is increasing on  $(-\infty,\log r_0]$ and
$\lim_{t\to -\infty} \mathcal G_{i_k}(t)=0$.  
Thus, $\mathcal G_{i_k}>0$ on $(-\infty,\log r_0]$, which implies that   
$$w_{i_k}'(t)>\left[(n-2)/2-\beta_0\right] w_{i_k}(t)\quad \text{for all } t\leq \log r_0.$$ 
Thus, $t\longmapsto  e^{-\left(\frac{n-2}{2}-\beta_0\right)t}w_{i_k}(t)$ is increasing on 
$(-\infty, \log r_0]$. Using \eqref{tr} and \eqref{r0}, we find that
\begin{equation} \label{ah} u_{i_k}(r)\leq c_0\, r^{-\beta_0}\quad \text{for every }r\in (0,r_0]\ \text{ and all }k\geq 1,\end{equation} where 
$c_0:=\left(\varepsilon_0/2\right) r_0^{-\frac{n-2}{2}+\beta_0} $. Since $\beta_0$ can be made arbitrarily small, it follows from \eqref{ah} that 
the right-hand side of \eqref{Eq0} with $u=u_{i_k}$ is uniformly bounded in  $L^p(B(0,r_0))$ for some $p>n/2$.
Then,   
$u_{i_k}$ satisfies \eqref{Eq0} in $\mathcal D'(B(0,r_0))$ (in the sense of distributions) and $(u_{i_k})_{k\geq 1}$ is uniformly bounded in $W^{2,p}(B(0,r_0))$ for some $p>n/2$. 
 Hence, $(u_{i_k}(r))_{k\geq 1}$ is uniformly bounded in $ r\in [0,r_0/2]$, which leads to a contradiction with $u_{i_k}(0)=\gamma_{i_k}\to \infty$ as $k\to \infty$. This ends the proof of Step~\ref{step52}. \qed
 \end{proof}
 
\noindent For $i\geq i_0$, we define 
 $$t_i:=\sup\left\{ t\in (-\infty, \log r_0):\ w_i(t)\geq \varepsilon_0\right\}.$$
It follows from Step \ref{step52} that $t_i$ is well-defined and that $t_i\in 
(-\infty,\log r_0)$ for all $i\geq i_0$.

\begin{step} \label{step53}
We claim that for every $i\geq i_0$, the function $w_i$ is decreasing on $[t_i,\bar t_i]$ for some 
$\bar t_i\in (t_i,\log r_0]$. Moreover, by diminishing $\varepsilon_0>0$ and $r_0>0$, there exist positive constants $c_1,c_2$ independent of $\varepsilon_0$ and $i$ such that 
\begin{equation} \label{sa}
 w_i(\bar t_i)\geq c_1 \varepsilon_0^{\frac{q+2}{2}} e^{\frac{\lambda t_i}{2}}
\quad \text{and}\quad
\bar t_i-t_i\leq \frac{2}{n-2} \log \frac{ \varepsilon_0}{w_i(\bar t_i)}
+c_2.
\end{equation}
Moreover, if $\bar t_i<\log r_0$, then $w_i$ is increasing on $[\bar t_i,\log r_0]$ and 
\begin{equation} \label{sa2}
\log r_0-\bar t_i\leq \frac{2}{n-2} \log \frac{w_i(\log r_0)}
{w_i(\bar t_i)}
+c_3,
\end{equation}
where $c_3>0$ is a constant independent of $\varepsilon_0$ and $i$.\end{step}

\begin{proof}[Proof of Step~\ref{step53}] Let $i\geq i_0$ be arbitrary. By Step~\ref{step52}, we have $w_i(t)\leq \varepsilon_0$ for every $t\in [t_i,\log r_0]$. Since \eqref{es} holds for all $t\in [t_i,\log r_0]$, as in the proof of Step~\ref{step52}, we regain \eqref{sec} replacing $w_{i_k}$ by $w_i$ for all $t\in [t_i,\log r_0]$. Hence, $t\longmapsto w_i'(t)$ is increasing on $[t_i,\log r_0]$ since $ w_i''(t)>0$ for all $t\in [t_i,\log r_0]$. 
We next distinguish two cases:

\vspace{0.2cm}
\noindent {\sc Case 1:} $w_i'(t)\not=0$ for all $t\in [t_i,\log r_0)$. Hence,  
  $w_i'<0$ on $[t_i,\log r_0)$ using that $w_i<\varepsilon_0$ on $(t_i,\log r_0]$. 

\vspace{0.2cm}
\noindent {\sc Case 2:}	$w_i'(\bar t_i)=0$ for some $\bar t_i\in [t_i,\log r_0)$. Then, $w_i'<0$ on $[t_i,\bar t_i)$ and $w_i'>0$ on $(\bar t_i,\log r_0]$.

\vspace{0.2cm} 
\noindent In both cases, $w_i$ is decreasing on $[t_i,\bar t_i]$ such that  
\begin{enumerate} 
	\item $\bar t_i=\log r_0$ in Case~1;
	\item $\bar t_i\in (t_i,\log r_0)$ and $w_i'(t)>0$ for all $t\in (\bar t_i,\log r_0]$ in Case~2.
	\end{enumerate} 

\vspace{0.2cm}
\noindent Unless explicitly mentioned, the argument below applies for both Case~1 (when $\bar t_i=
\log r_0$) and Case~2 (when $\bar t_i\in (t_i,\log r_0)$).  

\vspace{0.2cm}
\noindent From \eqref{weq}, we have that 
\begin{equation} \label{vv} w_i''(t)\geq f(w_i(t))\quad \text{for all } t\in [t_i,\log r_0]. 
\end{equation}
Thus, using \eqref{vv}, we find that
\begin{equation} \label{m0} t\longmapsto 
\left(w_i'(t)\right)^2-
F_0(w_i(t))
\end{equation}
\begin{enumerate}
	\item[(a)] is non-increasing on $[t_i,\bar t_i]$ (in both Case 1 and Case 2);
	
	\item[(b)] is non-decreasing on $[\bar t_i,\log r_0]$ in Case 2.  
\end{enumerate}

\vspace{0.1cm} 
{\em Proof of the first inequality in \eqref{sa}.} 
By \eqref{r0} and $w_i(t_i)=\varepsilon_0$, we infer that there exists $\widetilde t_i\in (t_i,\log r_0)$ such that $w_i(\widetilde t_i)=\varepsilon_0/2$ and, moreover, $\bar t_i\in (\widetilde t_i,\log r_0]$. Hence, there exists $\xi_i\in [t_i,\widetilde t_i]$ such that 
$$-\varepsilon_0/2=w_i(\widetilde t_i)-w_i(t_i)=w_i'(\xi_i)(\widetilde t_i-t_i).$$ 
By Step~\ref{step51}, $(|w_i'(t)|)_{i\geq 1}$ is uniformly bounded on $(-\infty, \log R_\Lambda)$ so that   
\begin{equation} \label{new} \widetilde t_i-t_i \geq c \varepsilon_0\quad \text{for some constant } c>0.\end{equation} 
From \eqref{ww}, \eqref{ei} and \eqref{eii}, there exists $\widetilde c>0$ such that
\begin{equation} \label{u1}
-\widetilde c \,w_i^2(t)\leq E_i(t)
\leq  E_i(\widetilde t_i)\quad \text{for every } \widetilde t_i<t \leq \log r_0. 
\end{equation}  
Moreover, using \eqref{new}, together with $E_i(t_i)<0$ and $w_i\geq \varepsilon_0/2$ on $[t_i,\widetilde t_i]$, we obtain that  
\begin{equation} \label{u2}
E_i(\widetilde t_i)=E_i(t_i)-\frac{2\lambda\mu}{q+1} 
\int_{t_i}^{\widetilde t_i} e^{\lambda t} w_i^{q+1}(t)\,dt\leq - \frac{ \lambda \mu c\, \varepsilon_0^{q+2} e^{\lambda t_i}}{2^{q}\left(q+1\right) } .   
\end{equation}     
Since $\bar t_i\in (\widetilde t_i, \log r_0]$, by combining \eqref{u1} and \eqref{u2}, there exists $c_1>0$ such that  
\begin{equation} \label{bb} 
w_i(\bar t_i)\geq c_1 \varepsilon_0^{\frac{q+2}{2}} \,e^{\frac{\lambda t_i}{2}}, 
\end{equation} 
where $c_1>0$ is independent of $\varepsilon_0$ and $i$. \qed

\vspace{0.2cm}
{\em Proof of the second inequality in \eqref{sa}.} 
From \eqref{m0}, for all $ t\in [t_i, \bar t_i)$, we have 
\begin{equation} \label{bbb} [w_i'(t)]^2- F_0(w_i(t)) \geq -F_0(w_i(\bar t_i)), \end{equation}
which jointly with $w_i'(t)<0$ and $F_0$ increasing on $[0,\epsilon_0]$, yields that 
\begin{equation} \label{hh}
-w_i'(t)\,\left[ F_0(w_i(t))
		-F_0(w_i(\bar t_i))\right]^{-1/2}\geq 1 \ \text{for all } t\in [t_i,\bar t_i).
\end{equation}
Hence, for all 
$t\in [t_i,\bar t_i)$, by integrating \eqref{hh} over $[t,\bar t_i]$, we get that 
\begin{equation} \label{not1} 
\bar t_i-t\leq \int_{w_i(\bar t_i)}^{w_i(t)} \frac{d\eta}{\left[F_0(\eta)-
		F_0(w_i(\bar t_i))\right]^{1/2}} =:
	\mathcal D_i(t).	\end{equation}
We shall prove below that 
\begin{equation} \label{pp}
\mathcal D_i(t)\leq \frac{2}{n-2} \left(\log \frac{w_i(t)}{w_i(\bar t_i)} +\log 2\right)
+ \tilde k w_i^{2^\star(s)-2}(t)
\end{equation} for all $t\in [t_i,\bar t_i)$, where $\tilde k>0$ is a constant independent of $\varepsilon_0$ and $i$. Then, since $w_i\leq \varepsilon_0$ on $[t_i,\bar t_i]$, from \eqref{not1} an \eqref{pp}, we conclude the proof of the second inequality in \eqref{sa}.   

\vspace{0.2cm}
{\em Proof of \eqref{pp}.} For every $\xi\geq 0$, we define 
\begin{equation} \label{m1} 
  g_i(\xi):=
\left(\frac{n-2}{2}\right)^2 \xi^2-
\frac{2}{2^\star(s)}\,\xi^{2^\star(s)}\left[
w_i(\bar t_i)\right]^{2^\star(s)-2}.
\end{equation}	
By a change of variable, we find that 
	\begin{equation} \label{di}
	\mathcal D_i(t)=
	\int_{1}^{
		w_i(t)/w_i
		(\bar t_i)} \frac{d\xi}{\left[ g_i(\xi)- g_i(1)\right]^{1/2}}\quad \text{for all } t\in [t_i,\bar t_i).
\end{equation} By the definition of $g_i$ in \eqref{m1}, for each $\xi>1$, we have 
\begin{equation} \label{m2} \frac{g_i(\xi)-g_i(1)}{\xi^2-1}= 
\left(\frac{n-2}{2}\right)^2 -\frac{2[w_i(\bar t_i)]^{2^\star(s)-2}}{2^\star(s)}  \frac{\xi^{2^\star(s)} -1}{\xi^2-1}.
\end{equation}
Since   
$ \left(2^\star(s)-1\right) \xi^{ 2^\star(s)}-2^\star(s)\, \xi^{2^\star(s)-2}+1$  increases for $ \xi\geq 1$,
we get that 
$\xi^{2^\star(s)} -1$ is bounded from above by $ 2^\star(s)\, \xi^{2^\star(s)-2}(\xi^2-1)$ for all $\xi\geq 1$. 
Hence, for any 
$1<\xi\leq \varepsilon_0/w_i(\bar t_i)$, we find that 
\begin{equation}  \label{m4}  \frac{[w( \bar t_i)]^{2^\star(s)-2}}{2^\star(s)} \frac{\xi^{2^\star(s)} -1}{\xi^2-1}\leq 
\left[ w_i(\bar t_i)\,\xi\right]^{2^\star(s)-2}\leq \varepsilon_0^{2^\star(s)-2}. 
\end{equation}
Since we fix $\varepsilon_0>0$ small, there exists a positive constant $k$, independent of $\varepsilon_0$ and $i$, such that 
\begin{equation} \label{st}
\left[\frac{\xi^2-1}{g_i(\xi)-g_i(1)}\right]^{1/2}\leq \frac{2}{n-2}+2k \left[w_i(\bar t_i)\,\xi\right]^{2^\star(s)-2}  
\end{equation}
for every $ 1<\xi\leq \varepsilon_0/w_i(\bar t_i)$.
Since $w_i(t)\leq w_i(t_i)\leq \varepsilon_0 $ for each $t\in [t_i,\bar t_i)$, using \eqref{st} in \eqref{di}, we get 
\begin{equation} \label{mm}
\mathcal D_i(t)\leq \frac{2}{n-2}\int_1^{h_i(t)}\frac{d\xi}{\left[\xi^2-1\right]^{1/2}}+2k \left[w_i(\bar t_i)\right]^{2^\star(s)-2} \mathcal E_i(t),
\end{equation}
where for every $t\in [t_i,\bar t_i)$, we define $h_i(t)$ and $\mathcal E_i(t)$ by 
\begin{equation} \label{jj}  
h_i(t):=\frac{w_i(t)}{w_i(\bar t_i)}
\quad \text{and} \quad {\mathcal E}_i(t):=\int_{1}^{h_i(t)}  \frac{\xi^{2^\star(s)-2}}
{\left(\xi^2-1\right)^{1/2}}\,d\xi.
\end{equation}
A simple calculation gives that there exists $C>0$ such that for every $t\in [t_i,\bar t_i) $, we have
\begin{equation} \label{int1}
  {\mathcal E}_i(t)\leq C
h_i^{2^\star(s)-2}(t).
\end{equation}
Using \eqref{jj} and \eqref{int1} into \eqref{mm}, we reach \eqref{pp} with $\tilde k$ large enough.
This completes the proof of the inequalities in \eqref{sa}. \qed

\vspace{0.2cm} {\em Proof of \eqref{sa2} in Case 2 (when $\bar t_i\in (t_i, \log r_0)$).} 
Recall that $w_i$ is increasing on $[\bar t_i, \log r_0]$ so that using \eqref{r0}, we get that $w_i(t)\leq w_i(\log r_0)<\varepsilon_0/2\quad \text{for all } t\in [\bar t_i, \log r_0].$ Moreover, 
the function in 
\eqref{m0} is non-decreasing on $[\bar t_i, \log r_0]$. Hence, we recover \eqref{bbb} for all $ t\in (\bar t_i, \log r_0]$. Since this time $w_i'>0$ on $(\bar t_i, \log r_0]$, instead of \eqref{not1}, we find that 
\begin{equation} \label{off} w_i'(t) \left[
	F_0(w_i(t))-F_0(w_i(\bar t_i))\right]^{-1/2}\geq 1\quad
\text{for every } t\in (\bar t_i,\log r_0]. 
\end{equation}
Using $\mathcal D_i(t)$ given by \eqref{not1}, we see that by integrating \eqref{off} over $[\bar t_i,t]$, we obtain that 
\begin{equation} \label{com}
t-\bar t_i\leq \mathcal D_i(t)\quad \text{for all } t\in (\bar t_i, \log r_0].
\end{equation}
Similar to the case $t\in [t_i,\bar t_i)$, we can prove 
\eqref{pp} for all $t\in (\bar t_i,\log r_0]$, which jointly with  \eqref{com}, gives the existence of a constant $c_3>0$ independent of $\varepsilon_0$ and $i$ such that 
\begin{equation} \label{ccc}
t-\bar t_i\leq \frac{2}{n-2}\log \frac{w_i(t)}{w_i(\bar t_i)}+c_3\quad
\text{for all } t\in (\bar t_i, \log r_0].
\end{equation} 
By letting $t=\log r_0$ in \eqref{ccc}, we conclude \eqref{sa2}. This proves the assertions of Step~\ref{step53}. \qed \end{proof}

 \begin{step} \label{step54}
 Proof of the Claim concluded in Case 1 of Step~\ref{step53}: $\bar t_i=\log r_0$. 
 \end{step}
 
 \begin{proof}[Proof of Step~\ref{step54}]
\noindent  Suppose that $w_i'<0$ on $ [t_i,\log r_0)$.

\vspace{0.2cm}
\noindent 
The second inequality in \eqref{sa} of Step~\ref{step53} reads as follows
\begin{equation} \label{ou}
\log r_0-t_i\leq \frac{2}{n-2}\log \frac{\varepsilon_0}{w_i(\log r_0)}+c_2.
\end{equation}  
The first inequality in \eqref{sa} and \eqref{r0} give that $ r_0^{\frac{n-2}{2}}u_0\geq c_1 \varepsilon_0^{\frac{q+2}{2}} \,e^{\frac{\lambda t_i}{2}}$. By applying $\log$ to this inequality  and to \eqref{bb} (with $\bar t_i=\log r_0$), respectively, we find that
\begin{equation} \label{ou1} \lambda t_i/2\leq 
[(n-2)/2] \log r_0+c_4
\left(\log u_0+\log (1/\varepsilon_0) \right) 
\end{equation}
for some constant $c_4>0$ independent of $\varepsilon_0$ and $i$, 
respectively 
\begin{equation} \label{ou2} \log (w_i(\log r_0)) \geq \lambda t_i/2+[(q+2)/2] \log \varepsilon_0 +\log c_1.
 \end{equation}
 Using \eqref{ou2} into \eqref{ou}, we deduce that 
\begin{equation} \label{ou3} \log r_0\leq \left[1-\lambda/(n-2)
 \right] t_i+c_5\log (1/\varepsilon_0) 
 \end{equation} for a constant
 $c_5>0$ independent of $\varepsilon_0$ and $i$. We have  
 \begin{equation} \label{theta} \Theta:= 2(q-2^\star+2)/(2^\star-1-q)>0\quad \text{since } q\in (2^\star-2,2^\star-1).\end{equation}
 Plugging into \eqref{ou3} the estimate on $t_i$ from \eqref{ou1}, we conclude that
 \begin{equation}\label{con} -\Theta \log r_0\leq c_6 \left[\log u_0 +\log (1/\varepsilon_0)\right],  
 \end{equation} where $c_6$ is a positive constant independent of $\varepsilon_0$ and $i$.   
 Since $\Theta>0$, we can choose $r_0>0$ small so that 
 the left-hand side of \eqref{con} is bigger than twice the right-hand side of \eqref{con}, 
which is a contradiction with \eqref{con}. This completes  Step~\ref{step54}. \qed \end{proof}

 \begin{step} \label{step55} 
 Proof of the Claim in Case~2 of Step~\ref{step53}: $\bar t_i\in (t_i,\log r_0)$. \end{step} 
  
  \begin{proof}[Proof of Step~\ref{step55}]  We have $w_i'<0$ on $[t_i,\bar t_i)$ and $w_i'>0$ on $(\bar t_i, \log r_0]$. 
 The first inequality of \eqref{sa} yields
   \begin{equation} \label{eee}
   2\log w_i(\bar t_i)\geq (q+2) \log \varepsilon_0+\lambda t_i+2\log c_1.
   \end{equation}
   By adding the second inequality of \eqref{sa} to that of \eqref{sa2}, we get
    \begin{equation} \label{aaa}
 \log r_0-t_i\leq \frac{2}{n-2} \left[\log 
 \varepsilon_0 +\log w_i(\log r_0)-2\log 
w_i(\bar t_i)
   \right]+ C_1,
      \end{equation} 
   where $C_1>0$ is a constant independent of $\varepsilon_0$ and $i$. 
 By \eqref{r0}, we have 
  \begin{equation} \label{abc}
 \log w_i(\log r_0)\leq \log u_0+[(n-2)/2]\,\log r_0.
   \end{equation}
 Using \eqref{eee} and \eqref{abc} into \eqref{aaa}, we obtain that 
\begin{equation} \label{cdcd}
  \left[ 2\lambda/(n-2)-1\right] t_i\leq C_2 \left[\log (1/\varepsilon_0)+\log u_0\right], 
\end{equation}
  where $C_2>0$ is a a constant independent of $\varepsilon_0$ and $i$. Since the coefficient of $t_i$ in \eqref{cdcd} equals $2^\star-2-q$, which is negative from the assumption $q>2^\star-2$, using that $t_i<\log r_0$, we infer that 
  \begin{equation} \label{ggg}
  \left(2^\star-2-q\right) \log r_0\leq C_2 \left[\log (1/\varepsilon_0)+\log u_0\right]. 
    \end{equation} 
  By choosing $r_0>0$ small so that the left-hand side of \eqref{ggg} is greater than twice the right-hand side of \eqref{ggg}, 
 we reach a contradiction. This proves Step~\ref{step55}. \qed
 \end{proof}

\noindent From Steps~\ref{step54} and \ref{step55} above, we conclude the proof of the Claim. \qed \end{proof}

\section{(CGS) solutions}\label{Sec5}

This section is devoted to the proof of part (iii) of Theorem~\ref{Th1}, restated below. 

\begin{proposition} \label{CGS} Let $q\in (1,\crit-1)$. There exists $R_0>0$ such that for every $R\in (0,R_0)$ and 
any positive singular solution $U$ of \eqref{rnsol}, there exists a unique positive radial (CGS) solution $u$ of \eqref{Eq0} with asymptotic profile $U$ near zero.
	\end{proposition}
	
\begin{proof} 
Let $f$ be given by \eqref{fdefi}. Let $U$ be a positive singular solution of \eqref{rnsol}. Then, by defining 
$\varphi(t)=e^{-(n-2)t/2}U(e^{-t})$ for $t\in \R$, we see that $\varphi\in C^\infty\(\R\)$ is a positive  
  periodic solution of 
 \begin{equation}\label{Th1CGSEq1}
 \varphi''(t)=f(\varphi(t))\quad \text{ for all } t\in \R.\end{equation} 
Let $\mathcal P$ denote the set of all positive smooth periodic solutions of  \eqref{Th1CGSEq1} to be described in Sect.~\ref{cgs61}.  
We next show that Proposition~\ref{CGS} is equivalent to Lemma~\ref{lem62}, the proof of which will be given in Sect.~\ref{cgs62}. 

\begin{lemma} \label{lem62} Let $q\in (1,\crit-1)$. For every $\varphi\in\mathcal P$, there exists $T_0=T_0(\varphi)>0$ large for which 
	the non-autonomous first order system 
	\begin{equation}\label{Th1CGSEq2}
	\left\{ \begin{aligned} 
	& (V',W')=(W,f(V)+\mu e^{-\lambda t} V^q)\qquad\text{in }[T_0,\infty),\\
	&  V>0\ \text{on } [T_0,\infty),\\
	\end{aligned} \right. \end{equation}
	has a unique solution satisfying 
	\begin{equation} \label{til}
	\( V(t),W(t)\)-\(\varphi(t),\varphi'(t)\)\to (0,0)\quad \text{as } t\to \infty.
	\end{equation}
\end{lemma}
\noindent Indeed, assuming that Proposition~\ref{CGS} holds, then for every $\varphi\in \mathcal P$, we use the transformation
\begin{equation} \label{vite} 
\varphi(t)=r^{\frac{n-2}{2}}U(r), \quad 
V(t)=r^{\frac{n-2}{2}} u(r),\quad W(t)=V'(t) \quad \text{with } t=\log  (1/r),\end{equation}
where $u$ is the unique positive radial (CGS) solution of \eqref{Eq0} satisfying $\lim_{r\to 0^+} u(r)/U(r)=1$. Hence, 
we obtain that $(V,W)$ is a solution of \eqref{Th1CGSEq2} for any $T_0>\log  \(1/R\)$ and, moreover, 
$V\(t\)-\varphi\(t\)\to0$ as $ t\to \infty$. 
Using \eqref{Th1CGSEq1}, we find that $W^\prime(t)-\varphi^{\prime\prime}(t)\to 0$ as $t\to \infty$.
Hence, $W-\varphi'$ is uniformly continuous on $[T_0,+\infty)$. Since $\lim_{t\to \infty} (V-\varphi)(t)= 0$, we get that $W\(t\)-\varphi'\(t\)\to 0$ as $t\to +\infty$. This proves Lemma~\ref{lem62}. 
We prove the reverse implication. If Lemma~\ref{lem62} holds, 
then for every positive singular solution $U$ of \eqref{rnsol}, by using 
\eqref{vite} and Proposition \ref{Pr}, we 
get a unique positive radial (CGS) solution $u$ of \eqref{Eq0} satisfying 
$\lim_{r\to 0^+} u(r)/U(r)=1$.

 \subsection{Description of $\mathcal P$} \label{cgs61} We show that the set $\mathcal P$ of all positive smooth periodic solutions of  \eqref{Th1CGSEq1} is given by \eqref{pers}. 
This is basically standard ODE theory. We state only the essential steps and leave the details to the reader. The function $F_0$ in \eqref{defyz} is increasing on $[0,M_0]$ and decreasing on $[M_0,\infty)$ with $M_0$ given by \eqref{mino}. 
 	Thus $F_0$ reaches its maximum $\overline \sigma$ at $M_0$, where 
 	\begin{equation} \label{mino}M_0:=\(\frac{n-2}{2}\)^{\frac{n-2}{2-s}} \quad \text{and} \quad \overline\sigma:=F_0(M_0)=\frac{2-s}{n-s}\(\frac{n-2}{2}\)^{\frac{2\(n-2\)}{2-s}}.\end{equation} 
 	Note that $M_0$ is the only positive zero of $f(\xi)=0$. Let $\varphi\in \mathcal P$. Since $F_0(\xi)=2\int_0^\xi f(t)\,dt$ for all $\xi\ge 0$, from \eqref{Th1CGSEq1}, there exists a constant $\sigma>0$ such that 
 	\begin{equation}  \label{pohs} \left(\varphi'(t)\right)^2=F_0(\varphi(t))- \sigma \quad \text{for all }t\in \R. 
 	\end{equation} In fact, by taking $\mu=0$ in \eqref{red} for $u=U$ with $U$ given by \eqref{vite}, we precisely obtain that 
 	$\sigma=2P(U)/\omega_{n-1}>0$, where $P(U)$ is the Pohozaev invariant associated to the positive singular solution $U$ of \eqref{rnsol}. 
 	From \eqref{mino} and \eqref{pohs}, we must have 
 	$$ 0<\sigma\leq {\overline \sigma}\quad {\rm and }\quad \varphi \equiv M_0 \  \text{ on } \R\ \text{ if } \sigma=\overline \sigma. $$
 	Let $\sigma\in (0,\overline \sigma)$ be fixed. Let $a_\sigma$ and $b_\sigma$ denote the two positive solutions of $F_0(\xi)=\sigma$ with $0<a_\sigma<M_0<b_\sigma$. It follows from standard analysis of the ODE \eqref{Th1CGSEq1} that for any $\sigma\in (0,\bar{\sigma})$, there is a unique solution $\varphi_\sigma$ to \eqref{Th1CGSEq1} such that $\min_\rr\varphi_\sigma=\varphi_\sigma(0)=a_\sigma<b_\sigma=
 	\max_{\rr}\varphi_\sigma$.  
 	Moreover, $\varphi_\sigma$ is periodic and we let $2t_\sigma>0$ be its principal period.

 	For every $\tau\in \mathbb S^1$, let $\varphi_{\sigma,\tau}$ denote the function whose graph is obtained from that of $\varphi_\sigma$ 
 	by a horizontal shift with  $(t_\sigma/\pi) \text{Arg }\,\tau $ units, where $\text{Arg} \,\tau$ denotes the principal argument of $\tau$. 
 	Note that $\varphi_\sigma=\varphi_{\sigma,\tau_0}$ with $\tau_0=(1,0)\in \mathbb S^1$. 
 	It follows that 
 	\begin{equation} \label{pers}  \mathcal P=\{\varphi_{\overline\sigma}\}\cup 
 	\{ \varphi_{\sigma,\tau}\}_{(\sigma,\tau)\in (0,\overline \sigma)\times \mathbb S^1},
 	\end{equation} 
 	where $\varphi_{\overline\sigma}\equiv M_0$ and 
 	$\varphi_{\sigma,\tau}(t)=\varphi_\sigma\left(t-(t_\sigma/\pi)
 	\text{Arg}\, \tau\right)$ for all $ t\in \R$.

\subsection{Proof of Lemma~\ref{lem62}} \label{cgs62}

We first prove Lemma~\ref{lem62} for $\varphi\in \cup \{\varphi_{\sigma,\tau}\}_{(\sigma,\tau) \in (0,\sigma_0)\times \mathbb S^1}$ with $\sigma_0\in (0,\overline\sigma)$ and second for $\varphi\in \{\varphi_{\overline \sigma}\}\cup \{ \varphi_{\sigma,\tau}\}_{(\sigma,\tau)\in [\sigma_0,\overline \sigma )\times \mathbb S^1}$ with $\sigma_0\in (0,\overline \sigma)$ close enough to $\overline \sigma$. 

\begin{step} \label{step621}For any $\sigma_*\in (0,\sigma_0)$ fixed, 
there exists $T_0>0$ large such that 
for every $\varphi=\varphi_{\sigma,\tau}$ with $\(\sigma,\tau\)\in (\sigma_*,\sigma_0)\times\mathbb S^1$, the system \eqref{Th1CGSEq2}, subject to 
\eqref{til}, admits a unique solution $(V_{\sigma,\tau},W_{\sigma,\tau})$. 
\end{step}

\begin{proof}[Proof of Step~\ref{step621}]
For the existence proof, we make a suitable transformation and use the Fixed Point Theorem for a contraction mapping. 
Let $I_0$ be an open interval such that $\(\sigma_*,\sigma_0\)\Subset I_0\Subset\(0,\overline{\sigma}\)$.  The key here is that for every $\(\sigma,\tau\)\in I_0\times\mathbb S^1$, both $\varphi_{\sigma,\tau}$ and $\partial_t\varphi_{\sigma,\tau}=\varphi'_{\sigma,\tau}$ are differentiable with respect to $\sigma$.  This does not hold for $\sigma=\overline\sigma$. By differentiating \eqref{pohs} with respect to $\sigma$ and using \eqref{Th1CGSEq1}, we get
$$  f(\varphi_{\sigma,\tau}(t))  \frac{d\[\varphi_{\sigma,\tau}(t)\]}{d\sigma}-\partial_t\varphi_{\sigma,\tau} (t) \frac{d \[\partial_t\varphi_{\sigma,\tau}(t)\]}{d\sigma}=\frac{1}{2}\quad 
\text{for all } t\in \R.
$$
We see that there exists $C_*>0$ such that for every $\(\sigma,\tau\)\in I_0\times\mathbb S^1$, we have
\begin{equation} \label{dstar} | \partial_t\varphi_{\sigma,\tau}(t)|+\left|   \frac{d\[\varphi_{\sigma,\tau}(t)\]}{d\sigma}\right| \leq C_*\quad \text{for all } t\in \R.
\end{equation}
Moreover, there exists $T_0>0$ such that $C_*  e^{-\lambda T_0/2}<a_0:=\inf\left\{a_\sigma:\,\sigma\in I_0\right\}$, where $a_\sigma$ is the smallest positive root of $F_0(\xi)=\sigma$. 
Let $\mathcal X_{T_0}$ denote the set of all continuous functions 
$(f_1,f_2):[T_0,\infty)\to \R^2$ with
$ e^{\lambda t/2} (|f_1(t)|+|f_2(t)|)\leq 1$ for all $ t\geq T_0$.
If we define $$\|(f_1,f_2)\|:=\sup_{t\geq T_0}\left\{ e^{\lambda t/2} \(|f_1(t)|+|f_2(t)|\)\right\},$$ then $(\mathcal X_{T_0},\|\cdot\|)$ is a complete metric space. 
For $(\widehat V,\widehat W)\in \mathcal X_{T_0}$ and recalling that $\varphi_{\sigma,\tau}''=f(\varphi_{\sigma,\tau})$ on $\R$, we consider the following transformation:
\begin{equation} \label{Th1CGSEq5}
\begin{bmatrix} 
	V(t)-\varphi_{\sigma,\tau}(t) \\ W(t)-\varphi_{\sigma,\tau}'(t)
	\end{bmatrix} =\begin{bmatrix} \partial_t\varphi_{\sigma,\tau} (t)& \frac{d\[\varphi_{\sigma,\tau}(t)\]}{d\sigma}\\ 
	f(\varphi_{\sigma,\tau}(t)) &  \frac{d \[\partial_t\varphi_{\sigma,\tau}(t)\]}{d\sigma}
		\end{bmatrix}
	\begin{bmatrix} \widehat V(t)\\ \widehat W(t)
	\end{bmatrix}\quad \text{for } t\in [T_0,\infty).
\end{equation}
Note that the matrix and its inverse are both uniformly bounded with respect to $\(\sigma,\tau\)\in I_0\times\mathbb S^1$. In particular, \eqref{Th1CGSEq5} yields that
\begin{equation} \label{vii} V(t):= \varphi_{\sigma,\tau}(t)+\widehat V (t)\partial_t\varphi_{\sigma,\tau}(t)+\widehat W(t) \frac{d\[\varphi_{\sigma,\tau}(t)\]}{d\sigma}.
\end{equation}
From \eqref{dstar}, \eqref{vii} and our choice of $T_0$, we find that
\begin{equation} \label{inn}  |V(t)-\varphi_{\sigma,\tau}(t)|\leq C_*  e^{-\lambda t/2} 
\| (\widehat V,\widehat W)\|
 \leq C_* e^{-\lambda T_0/2} <a_0\quad \text{for all } t\geq T_0.
\end{equation} 
For every $t\geq T_0$ and $\(\sigma,\tau\)\in I_0\times\mathbb S^1$, we have 
$\varphi_{\sigma,\tau}(t)\geq a_{\sigma}\geq a_0$ since $a_{\sigma}$ is increasing in $\sigma$.
Thus, \eqref{inn} proves that $V$ in \eqref{vii} is positive on $[T_0,\infty)$ for all $(\widehat V,\widehat W)\in \mathcal X_{T_0}$. 
For simplicity of reference, using $V$ in \eqref{vii} for $(\widehat V,\widehat W)\in \mathcal X_{T_0}$, we define 
$$ g_{\sigma,\tau,\widehat V,\widehat W}(t):=f(V(t))-f(\varphi_{\sigma,\tau}(t)) -(V(t)-\varphi_{\sigma,\tau}(t)) f'(\varphi_{\sigma,\tau}(t))+\mu e^{-\lambda t} V^q(t).
$$
By \eqref{inn}, there exist positive constants $C_0$ and $C_1$ such that for all
 $(\widehat V,\widehat W)\in \mathcal X_{T_0}$, 
 \begin{equation} \label{gb} |g_{\sigma,\tau,\widehat V,\widehat W}(t)|\leq C_0   |V(t)-\varphi_{\sigma,\tau}(t)|^2+\mu e^{-\lambda t} V^q(t)\leq C_1 e^{-\lambda t}
 \end{equation} for every $t\in [T_0,\infty)$ and $\(\sigma,\tau\)\in I_0\times\mathbb S^1$.
Remark that \eqref{Th1CGSEq2} is equivalent to the system
\begin{equation} \label{Th1CGSEq6}
( \widehat V'(t),\widehat W'(t))= \left( 2g_{\sigma,\tau,\widehat V,\widehat W}(t)  \frac{d\[\varphi_{\sigma,\tau}(t)\]}{d\sigma},- 2 g_{\sigma,\tau,\widehat V,\widehat W}(t)  \,\partial_t\varphi_{\sigma,\tau}(t) \right) \  \text{on } [T_0,\infty) .
\end{equation}
For every $(\widehat V,\widehat W)\in \mathcal X_{T_0}$ and $t\geq T_0$, we define  
$$  \Phi_{\sigma,\tau}(\widehat V,\widehat W)(t)=
\left(-2\int_t^\infty g_{\sigma,\tau,\widehat V,\widehat W}(y) \frac{d\[\varphi_{\sigma,\tau}(y)\]}{d\sigma}\,dy,
2\int_t^\infty g_{\sigma,\tau,\widehat V,\widehat W}(y) \,\partial_t\varphi_{\sigma,\tau}(y)\,dy\right). 
$$
We next prove the existence of $T_0>0$ large such that  
$\Phi_{\sigma,\tau}$ maps $\mathcal X_{T_0}$ into $\mathcal X_{T_0}$ and $\Phi_{\sigma,\tau}$ is a contraction mapping on $\mathcal X_{T_0}$ for every $\(\sigma,\tau\)\in I_0\times\mathbb S^1$. 
From \eqref{dstar}, \eqref{gb} and the definition of $(\mathcal X_{T_0},\|\cdot\|)$, we have
\begin{equation} \label{phm}
\|\Phi_{\sigma,\tau} (\widehat V,\widehat W)\|
 \leq 2C_* \sup_{t\geq T_0} \left\{ e^{\lambda t/2} \int_t^\infty |g_{\sigma,\tau,\widehat V,\widehat W}(y)|\,dy\right\}\leq \frac{2 C_* C_1}{\lambda} e^{-\lambda T_0/2}.
\end{equation}
Thus, for large $T_0>0$, we find that $\Phi_{\sigma,\tau} (\widehat V,\widehat W)\in \mathcal X_{T_0}$ for every $(\widehat V,\widehat W)\in \mathcal X_{T_0}$ and all $\(\sigma,\tau\)\in I_0\times\mathbb S^1$. For every $(\widehat V_1,\widehat W_1)$ and $(\widehat V_2,\widehat W_2)$ in $ \mathcal X_{T_0}$, we have 
$$ \|(\widehat V_1,\widehat W_1)- (\widehat V_2,\widehat W_2)\|= \sup _{t\geq T_0}  \left\{ e^{\frac{\lambda t}{2}} \widehat S(t) \right\}\text{ 
with } \widehat S(t):=|(\widehat V_1-\widehat V_2)(t)|+
|(\widehat W_1-\widehat W_2)(t)| .$$ Hence, there exist positive constants $C_2$ and $C_3$ such that for every $\(\sigma,\tau\)\in I_0\times\mathbb S^1$ 
$$ \begin{aligned}
 \| \Phi_{\sigma,\tau}(\widehat V_1,\widehat W_1)-\Phi_{\sigma,\tau}(\widehat V_2,\widehat W_2) \| 
 & \leq 2C_* \sup_{t\geq T_0} \left\{ e^{\frac{\lambda t}{2}}\int_t^\infty \left|g_{\sigma,\tau,\widehat V_1,\widehat W_1}(y)-
g_{\sigma,\tau,\widehat V_2,\widehat W_2}(y) \right|\,dy \right\}\\
& \leq C_2 \sup_{t\geq T_0}  \left\{ e^{\frac{\lambda t}{2}} \int_t^\infty \left[ 
(\widehat S(y))^2
+e^{-\lambda y} \widehat S(y)\right]\,dy\right\}\\
& \leq C_3 e^{-\frac{\lambda T_0 }{2}}\| (\widehat V_1,\widehat W_1)- (\widehat V_2,\widehat W_2) \|
\end{aligned} $$
for all $(\widehat V_1,\widehat W_1),(\widehat V_2,\widehat W_2)\in  \mathcal X_{T_0}$. Thus, for $T_0>0$ large, $\Phi_{\sigma,\tau}$ is a contraction on $ \mathcal X_{T_0}$ for all $\(\sigma,\tau\)\in I_0\times\mathbb S^1$.  
Hence, $\Phi_{\sigma,\tau}$ has a unique fixed point in $\mathcal X_{T_0}$, say $(\widehat V_{\sigma,\tau},\widehat W_{\sigma,\tau})$, which gives a unique solution in $\mathcal X_{T_0}$ of \eqref{Th1CGSEq6} such that $\lim_{t\to \infty}(\widehat V_{\sigma,\tau},\widehat W_{\sigma,\tau})(t)= (0,0)$. By \eqref{dstar} and $(\widehat V,\widehat W)=(\widehat V_{\sigma,\tau},\widehat W_{\sigma,\tau})$ in \eqref{Th1CGSEq5}, we obtain a solution $(V_{\sigma,\tau},W_{\sigma,\tau})$ of \eqref{Th1CGSEq2} satisfying \eqref{til} with $\varphi=\varphi_{\sigma,\tau}$. Moreover, $(V_{\sigma,\tau},W_{\sigma,\tau})$ is continuous in $\(\sigma,\tau\)$ since $\Phi_{\sigma,\tau}$ is continuous in $\(\sigma,\tau\)$. 

\smallskip\noindent To prove uniqueness, on $\Omega_0:=I_0\times\mathbb S^1\times(0,e^{-T_0})$, we define the functions $H,G:\Omega_0\to\R^3$ by 
$$H\(\sigma,\tau,r\):=\(V_{\sigma,\tau}\(t\(r\)\),W_{\sigma,\tau}\(t\(r\)\),r\)\hbox{ and }G\(\sigma,\tau,r\):=\(\varphi_{\sigma,\tau}\(t\(r\)\),\varphi'_{\sigma,\tau}\(t\(r\)\),r\)$$
for every $\(\sigma,\tau,r\)\in\Omega_0$, where $t\(r\):=\log \(1/r\)$. From our construction, $H$ is continuous. Since $V_{\sigma,\tau}$ is a solution to a second-order ODE and $W_{\sigma,\tau}=V_{\sigma,\tau}'$, the uniqueness theorem for ODEs yields that $H$ is one-to-one in $\Omega_0$. Clearly, $G$ is also continuous and one-to-one in $\Omega_0$. Thus, by applying the Domain Invariance Theorem, we obtain that $H$ and $G$ are open. Moreover, since the functions $\left\{\varphi_{\sigma,\tau}\right\}_{\(\sigma,\tau\)\in I_0\times\mathbb S^1}$ are periodic, we see that $G\(\Omega_0\)=\Sigma_0\times(0,e^{-T_0})$ for some domain $\Sigma_0$ in $\R^2$. Let $H_0:\Sigma_0\times(-e^{-T_0},e^{-T_0})\to\R^3$ be the function defined as
$$H_0\(\xi_1,\xi_2,r\):=\left\{\begin{aligned}&H\(G^{-1}\(\xi_1,\xi_2,r\)\)&&\text{if }r>0\\&\(\xi_1,\xi_2,0\)&&\text{if }r=0\\&J\(H\(G^{-1}\(\xi_1,\xi_2,-r\)\)\)&&\text{if }r<0,\end{aligned}\right.$$
where $J (\xi_1,\xi_2,r):=(\xi_1,\xi_2,-r)$. Since $H$ and $G$ are one-to-one in $\Omega_0$, we obtain that $H_0$ is one-to-one in $\Sigma_0\times(-e^{-T_0},e^{-T_0})$. Moreover, since $H$ and $G^{-1}$ are continuous in $\Omega_0$, we obtain that $H_0$ is continuous in $\Sigma_0\times[(-e^{-T_0},e^{-T_0})\backslash\left\{0\right\}]$. As regards the continuity of $H_0$ on $\Sigma_0\times\left\{0\right\}$, for every $\(\zeta_1,\zeta_2\)\in\Sigma_0$ and $\(\xi_1,\xi_2,r\)\in\Sigma_0\times[(-e^{-T_0},e^{-T_0})\backslash\left\{0\right\}]$, we write
\begin{multline}\label{Continuity}
\left|H_0\(\xi_1,\xi_2,r\)-H_0\(\zeta_1,\zeta_2,0\)\right|\le\left|H_0\(\xi_1,\xi_2,r\)-\(\xi_1,\xi_2,r\)\right|+\left|\(\xi_1,\xi_2,r\)-\(\zeta_1,\zeta_2,0\)\right|\\
\le\left|\(V_{\sigma,\tau}\(t\(\left|r\right|\)\)-\varphi_{\sigma,\tau}\(t\(\left|r\right|\)\),W_{\sigma,\tau}\(t\(\left|r\right|\)\)-\varphi'_{\sigma,\tau}\(t\(\left|r\right|\)\)\)\right|+\left|\(\xi_1,\xi_2,r\)-\(\zeta_1,\zeta_2,0\)\right|,
\end{multline}
where $\(\sigma,\tau,\left|r\right|\)=G^{-1}\(\xi_1,\xi_2,\left|r\right|\)$. Since $(\widehat V_{\sigma,\tau},\widehat W_{\sigma,\tau})\in \mathcal X_{T_0}$, it follows from \eqref{dstar}, \eqref{Th1CGSEq5} and \eqref{Continuity} that for every $\(\zeta_1,\zeta_2,0\)\in\Sigma_0\times\left\{0\right\}$, $H_0\(\xi_1,\xi_2,r\)\to H_0\(\zeta_1,\zeta_2,0\)$ as $\(\xi_1,\xi_2,r\)\to\(\zeta_1,\zeta_2,0\)$ i.e., $H_0$ is continuous at $\(\zeta_1,\zeta_2,0\)$. By another application of the Domain Invariance Theorem, we obtain that $H_0$ is open. We let $\Sigma_*$ be the domain such that $\Sigma_*\times\left\{r\right\}=G(\(\sigma_*,\sigma_0\)\times\mathbb S^1,r
)$ for every $r>0$. In particular, since $\Sigma_*$ is open, we obtain that for every $\(\sigma,\tau\)\in (\sigma_*,\sigma_0)\times\mathbb S^1$, every solution $(V\(t\),W\(t\))$ of \eqref{Th1CGSEq2} satisfying
\begin{equation} \label{til2}
\( V(t),W(t)\)-\(\varphi_{\sigma,\tau}(t),\varphi_{\sigma,\tau}'(t)\)\to (0,0)\quad \text{as } t\to \infty
\end{equation}
must satisfy $(X(t\(r\)),Y(t\(r\)))\in \Sigma_*$ for small $r>0$. Since $\Sigma_*\times\left\{0\right\}\Subset H_0(\Sigma_0\times(-e^{-T_0},e^{-T_0}))$ and $H_0$ is open, we obtain that there exists $R_0\in(0,e^{-T_0})$ such that $\Sigma_*\times \[-R_0,R_0\]\subset H_0(\Sigma_0\times(-e^{-T_0},e^{-T_0}))$. It then follows from the definitions of $H_0$ and $\Sigma_*$ that $\Sigma_*\times\(0,R_0\]\subset H(I_0\times\mathbb S^1\times(0,R_0])$. Hence, for every solution $(X\(t\),Y\(t\))$ of \eqref{Th1CGSEq2} satisfying \eqref{til2} for some $\(\sigma,\tau\)\in (\sigma_*,\sigma_0)\times\mathbb S^1$, we obtain $(X(t\(r\)),Y(t\(r\)),r)\in H(I_0\times\mathbb S^1\times(0,R_0])$ and so $(X(t\(r\)),Y(t\(r\)))=(V_{\sigma,\tau}(t\(r\)),W_{\sigma,\tau}(t\(r\)))$ for small $r>0$. Hence, for every $\varphi=\varphi_{\sigma,\tau}$ with $\(\sigma,\tau\)\in(\sigma_*,\sigma_0)\times\mathbb S^1$, we conclude that 
$(V_{\sigma,\tau},W_{\sigma,\tau})$ is the unique solution of \eqref{Th1CGSEq2} satisfying \eqref{til}. This ends Step~\ref{step621}. \qed
\end{proof}

\begin{step} \label{step622}
Proof of Lemma~\ref{lem62} if 
$\varphi\in \cup \{ \varphi_{\sigma,\tau}\}_{(\sigma,\tau)\in [\sigma_0,\overline \sigma ]\times {\mathbb S}^1}$ for $\sigma_0\in (0,\overline \sigma)$ close enough to $\overline \sigma$. \end{step}

\begin{proof}[Proof of Step~\ref{step622}] For $(\sigma,\tau)\in (0,\overline\sigma]\times \mathbb{S}^1$, we search for $T_0\in\mathbb{R}$ and $V,W\in C^1([T_0,+\infty))$ such that \eqref{Th1CGSEq2} holds and $(V(t),W(t))-(\varphi_{\sigma,\tau}(t),\varphi_{\sigma,\tau}'(t))\to (0,0)$ as $t\to\infty$. Writing $\tilde{V}=V-\varphi_{\sigma,\tau}$ and $\tilde{W}=W-\varphi_{\sigma,\tau}'$, this is equivalent to finding $T_0\in\mathbb{R}$ and $\tilde{V},\tilde{W}\in C^1([T_0,+\infty))$ such that 
\begin{equation}\label{syst:tilde}
\left\{ \begin{aligned}
& (\tilde{V}',\tilde{W}')=(\tilde{W},f(\varphi_{\sigma,\tau}+\tilde{V})-f(\varphi_{\sigma,\tau})+\mu e^{-\lambda t}(\varphi_{\sigma,\tau}+\tilde{V})^q )
\quad \text{in } [T_0,\infty),\\
&  (\tilde{V}(t),\tilde{W}(t))\to (0,0)\ \ \text{as } t\to +\infty. 
\end{aligned}\right.
\end{equation}
We define $L(\varphi_{\sigma,\tau},\varphi_{\overline\sigma})(t):=f'(\varphi_{\sigma,\tau}(t))-f'(\varphi_{\overline\sigma}(t))$ for $t\in \R$ and 
$$A:=\left(\begin{array}{cc} 0 & 1\\ f'(M_0) & 0\end{array}\right).$$
Since $\varphi_{\overline\sigma}\equiv M_0$ and $\varphi_{\sigma,\tau}\to\varphi_{\overline\sigma}$ as $\sigma\to\overline\sigma$ uniformly with respect to $\tau\in \mathbb{S}^1$, we get that
\begin{equation}\label{lim:H}
\lim_{\sigma\to\overline\sigma}\sup_{\tau\in \mathbb{S}^1,\, t\in\mathbb{R}}|L(\varphi_{\sigma,\tau},\varphi_{\overline\sigma})(t)|=0.
\end{equation}
With a Taylor expansion, we write
\begin{eqnarray*} 
f(\varphi_{\sigma,\tau}+\tilde{V})-f(\varphi_{\sigma,\tau})&=&f'(\varphi_{\overline\sigma})\tilde{V}+L(\varphi_{\sigma,\tau},\varphi_{\overline\sigma})\tilde{V}+Q(\varphi_{\sigma,\tau},\tilde{V})
\end{eqnarray*}
with $|Q(\varphi_{\sigma,\tau},\tilde{V})|\leq C |\tilde{V}|^2$. Therefore, the system in \eqref{syst:tilde} rewrites as follows
\begin{equation}\left(\begin{array}{c}
\tilde{V}\\
\tilde{W}\end{array}\right)^\prime=A\left(\begin{array}{c}
\tilde{V}\\
\tilde{W}\end{array}\right)+\left(\begin{array}{c}
0\\
L(\varphi_{\sigma,\tau},\varphi_{\overline\sigma})\tilde{V}+Q(\varphi_{\sigma,\tau},\tilde{V})+\mu e^{-\lambda t}(\varphi_{\sigma,\tau}+\tilde{V})^q\end{array}\right). 
\end{equation}
Since $f'(M_0)<0$, we get that $A$ has two conjugate pure imaginary eigenvalues. Therefore, there exists $C>0$ such that $\Vert e^{tA}\Vert+\Vert e^{-tA}\Vert\leq C$ for all $t\in\mathbb{R}$,
where $\Vert\cdot\Vert$ is any operator norm on $\mathbb{R}^2$. For all $t\geq T_0$, we define $\tilde{X}(t):=e^{-t A}\left(\begin{array}{c}
\tilde{V}(t)\\
\tilde{W}(t)\end{array}\right)$ 
and $$\Phi_{\varphi_{\sigma,\tau}}(t,\tilde{X}):=e^{-tA}\left(\begin{array}{c}
0\\
L(\varphi_{\sigma,\tau},\varphi_{\overline\sigma})(e^{tA}\tilde{X})_1+Q(\varphi_{\sigma,\tau},(e^{tA}\tilde{X})_1)+\mu e^{-\lambda t}(\varphi_{\sigma,\tau}+(e^{tA}\tilde{X})_1)^q\end{array}\right),
$$ where $(e^{tA}\tilde{X})_1$ denotes the first coordinate of $e^{tA}\tilde{X}\in \mathbb{R}^2$. 
Then getting a solution to \eqref{syst:tilde} 
amounts to finding a solution $\tilde{X}\in C^1([T_0,+\infty),\mathbb{R}^2)$ to
\begin{equation}\label{syst:tilde:2}
\tilde{X}'(t)=\Phi_{\varphi_{\sigma,\tau}}(t,\tilde{X}(t))\hbox{ for }t\geq T_0\hbox{ and }\lim_{t\to +\infty}\tilde{X}(t)=0.
\end{equation}
As in Step \ref{step621}, we find a solution to \eqref{syst:tilde:2} via the Fixed Point Theorem for contracting maps on a complete metric space. Since $Q(\varphi_{\sigma,\tau},\tilde{V})$ is quadratic in $\tilde{V}$, the last two terms of the second coordinate of  $\Phi_{\varphi_{\sigma,\tau}}(t,\tilde{X})$ are tackled as in Step \ref{step621}. The first term is linear in $\tilde{V}$ and controled by  $L(\varphi_{\sigma,\tau},\varphi_{\overline\sigma})$: with \eqref{lim:H}, this term is contracting for $\sigma$ close enough to $\overline\sigma$. Mimicking the existence proof of Step \ref{step621}, we get the following:

\smallskip\noindent There exist $\epsilon>0$ and $T_0>0$ such that for every $(\sigma,\tau)\in [\overline\sigma-3\epsilon,\overline\sigma]\times 
\mathbb{S}^1$, there exists a solution $(V_{\sigma,\tau},W_{\sigma,\tau})\in C^1([T_0,+\infty),\mathbb{R}^2)$ to \eqref{Th1CGSEq2} such that \eqref{til} holds for $\varphi=\varphi_{\sigma,\tau}$. Moreover, since $(\sigma,\tau)\longmapsto (\varphi_{\sigma,\tau},\varphi_{\sigma,\tau}')$ is continuous on $(0,\overline\sigma]\times \mathbb{S}^1$ (despite the issue for $\overline\sigma$), the continuity of the fixed points depending on a parameter yields that $(\sigma,\tau)\longmapsto (V_{\sigma,\tau},W_{\sigma,\tau})$ is continuous on $[\overline\sigma-3\epsilon,\overline\sigma]\times \mathbb{S}^1$. Here we have taken the supremum norm on $C^0([T_0,+\infty),\mathbb{R}^2)$: via the fixed point construction, we also get that this holds with a weighted norm.

\smallskip\noindent We only sketch the uniqueness proof. For $\tau_0=(1,0)\in \mathbb{S}^1$ and 
every $\xi\in B(0,2\epsilon)\subset \mathbb{R}^2$, 
we define 
$$\sigma(\xi):=\overline\sigma-|\xi|\hbox{ and }\left\{\tau(\xi):=\xi/|\xi|\hbox{ if }\xi\neq 0\hbox{ and }\tau(0)=\tau_0\right\}. $$
Due to the uniqueness of solution for $\sigma=\overline{\sigma}$, as one checks, we have the continuity of the mappings $\xi\longmapsto (\varphi_{\sigma(\xi),\tau(\xi)},\varphi_{\sigma(\xi),\tau(\xi)}')$ and $\xi\longmapsto (V_{\sigma(\xi),\tau(\xi)},W_{\sigma(\xi),\tau(\xi)}')$ on $B(0,2\epsilon)$. We introduce the domain $\Omega_0:=B(0,2\epsilon)\times(0,e^{-T_0})$ and the functions $H,G:\Omega_0\to\R^3$ defined as 
$$H\(\xi,r\):=\(V_{\sigma,\tau}\(t\),W_{\sigma,\tau}\(t\),r\)\hbox{ and }G\(\xi,r\):=\(\varphi_{\sigma,\tau}\(t\),\varphi'_{\sigma,\tau}\(t\),r\)$$
for every $\(\xi,r\)\in\Omega_0$, where $t\(r\):=\log \(1/r\)$, $\sigma=\sigma(\xi)$ and $\tau=\tau(\xi)$. Arguing as in Step \ref{step621} and with some extra care for the case $\xi=0$, we get the uniqueness of the solution of \eqref{Th1CGSEq2} satisfying \eqref{til} for $\varphi=\varphi_{\sigma,\tau}$. This ends the proof of Step~\ref{step622}. \qed
\end{proof}

This completes the proof of Lemma~\ref{lem62} and thus of Proposition~\ref{CGS}. \qed
\end{proof}

\section{Appendix} \label{appen}

Here, we establish Theorem~\ref{71}, a critical result that was used in the proof of Lemma~\ref{exist:sol}. 
The proof of Theorem~\ref{71} is strongly inspired by Kelley's paper \cite{Kel}. 
We denote by $B_\delta(0)\subset \R^3$ the ball centered at $0$ with radius $\delta>0$. For any $r_0>0$, we set $D_{r_0}:=[0,r_0]\times [-r_0,r_0]$. 

\begin{theorem} \label{71} Let $h_j\in C^1(B_\delta(0))$ for some $\delta>0$ with $j=1,2,3$. Suppose there exist constants $C_1>0$ and $p>1$ such that for all $\vec \xi=(\xi_1,\xi_2,\xi_3)\in B_\delta(0)$, we have
	\begin{equation} \label{hy}\left\{ \begin{aligned}
	&\sum_{j=1}^3 |h_j(\vec \xi)|\leq C_1\sum_{j=1}^3 \xi_j^2\ \text{and}\ \
	\sum_{j=1}^3  |\nabla h_j(\vec \xi)|\leq C_1\sum_{j=1}^3 |\xi_j|,\\
	& h_2(\vec \xi)\leq -C_1|\xi_2|^p\ \text{and }
	h_2(\xi_1,0,\xi_3)=0.
	\end{aligned} \right.\end{equation}
For fixed $a>0$ and 
	$c<0$, we consider the system
	\begin{equation}\label{syst:th:bis}
\left\{\begin{aligned}
	&\vec {\mathcal Z}'(t)=(a{\mathcal Z}_1(t)+h_1(\vec {\mathcal Z}(t)),h_2(\vec {\mathcal Z}(t)),c{\mathcal Z}_3(t)+
	h_3(\vec {\mathcal Z}(t)))\quad \text{for }t\geq 0,\\
	&  \vec {\mathcal Z}(0)=(x_0,y_0,z_0).
\end{aligned}	\right.
	\end{equation}
	Then there exist $r_0\in (0,\delta/2)$ and a Lipschitz function $w: D_{r_0}\to [-r_0,r_0]$ such that for all $(y_0,z_0)\in D_{r_0}$ and $x_0=w(y_0,z_0)$, the initial value system \eqref{syst:th:bis} has a solution $\vec {\mathcal Z}$ on $[ 0,\infty)$ and 
	\begin{equation} \label{mn}\lim_{t\to +\infty}  \vec {\mathcal Z}(t)=(0,0,0).\end{equation}
	Moreover, we have that the parametrized surface $({\mathcal Z}_2,{\mathcal Z}_3)\longmapsto (w({\mathcal Z}_2,{\mathcal Z}_3),{\mathcal Z}_2,{\mathcal Z}_3)$ is stable in the sense that ${\mathcal Z_1}(t)=w({\mathcal Z}_2(t),{\mathcal Z}_3(t))$ for all $t\geq 0$. \end{theorem}

\begin{proof} 
Since $h_j\in C^1(B_\delta(0))$ for $1\leq j\leq 3$, the Cauchy--Lipschitz theory applies to the system. For $r_0\in (0,\delta/2)$ and $C_2>0$, we define $\mathcal X$ as the set of all continuous functions $w:D_{r_0}\to [-r_0,r_0]$ such that $w(0,0)=0$ and $w$ is $C_2$-Lipschitz. 
Note that $(\mathcal X,\Vert\cdot\Vert_\infty)$ is a complete metric space. 
For any $w\in \mathcal X$, we consider the system
\begin{equation} \label{syst:w} \tag{$S_w$}
\left\{ \begin{aligned}
& (y',z')= \( h_2 (w(y,z),y,z), cz+h_3(w(y,z),y,z) \) \quad \text{on } [0,\infty),\\ 
& (y(0),z(0))=(y_0,z_0).
\end{aligned} \right.
\end{equation}

We now divide the proof of Theorem~\ref{71} in five Steps.

\begin{step} \label{step71}
Let $r_0\in (0,\delta/2)$ be such that $4C_1(1+C_2^2) \,r_0\leq |c|$. If $(y_0,z_0)\in D_{r_0}$, then the flow $\Phi_t^w(y_0,z_0)$ associated to \eqref{syst:w} is defined for all $t\in [0,+\infty)$. If we set 
\begin{equation} \label{yzx} (y(t),z(t)):=\Phi_t^w(y_0,z_0)\quad \text{for all }t\in [0,\infty),\end{equation} then 
$ 0\leq y(t)\leq y_0$ and $ |z(t)|\leq \max\{y_0,|z_0|\}$ on $[0,\infty)$. Moreover, we have
\begin{equation} \label{ac} \lim_{t\to \infty} (y(t),z(t))=(0,0).\end{equation}
\end{step}

\begin{proof}[Proof of Step~\ref{step71}]
Let $(y_0,z_0)\in D_{r_0}$ be arbitrary. Since the Cauchy--Lipschitz theory applies, the initial value problem \eqref{syst:w} has a 
unique solution $(y,z) $ 
on an interval $[0,b)$ with $b>0$. We prove the following:
\begin{itemize} \item[(i)] $y\equiv 0$ if $y_0=0$ and 
$0<y(t)\leq y_0$ for all $t\in [0,b)$ if $y_0\in (0,r_0]$;
\item[(ii)] $|z(t)|\leq \max\{y_0,|z_0|\}$ for every $t\in [0,b)$.  
\end{itemize}

\begin{proof}[Proof of (i)] We write $h_2 (w(y(t),z(t)),y(t),z(t))=\widehat h_2(t,y(t))$ for $t\in [0,b)$, where $\widehat h_2(t,y)$ is continuous in $t\in [0,b)$ and Lipschitz with respect to $y\in [0,r_0]$. The assumption \eqref{hy} yields
$\widehat h_2(\cdot,0)=0$ on $[0,b)$ and $\widehat  h_2(t,y(t))\leq 0$ for all $t\in [0,b)$. The claim of (i) holds since $y'(t)\leq 0$ on $[0,b)$ and $y$ is the unique solution of $y'(t)=\widehat h_2(t,y(t))$ for $t\in [0,b)$, subject to  $y(0)=y_0$. \qed \end{proof}

\begin{proof}[Proof of (ii)] 
Since $c<0$, using the system $(S_w)$, we find that
\begin{equation}\label{eq:zdeux}
(z^2)'=2\left(-|c|z^2+z \,h_3(w(y,z),y,z)\right)\quad \text{on } [0,b).
\end{equation}
Since $w$ is a $C_2$-Lipschitz function, using the hypothesis on $h_3$ in \eqref{hy}, we have
\begin{equation}\label{control:h}
\left| z\,h_3(w(y,z),y,z)\right|\leq C_1 |z| \left[w^2(y,z)+y^2+z^2\right]\leq C_1\(1+C_2^2\)|z|\(y^2+z^2\). 
\end{equation}
Using $|z|\leq r_0$ and the choice of $r_0>0$, from \eqref{control:h} we obtain that
\begin{equation} \label{nuc}
\left| z\,h_3(w(y,z),y,z)\right|\leq   |c|\max \{y^2,z^2\}/2\quad \text{on } [0,b).
\end{equation}

If $y_0=0$, then $y\equiv 0$ on $[0,b)$ by (i). From \eqref{nuc} and \eqref{eq:zdeux}, we have $(z^2)'\leq 0$ on $[0,b)$, which yields
$|z(t)|\leq |z_0|$ for all $t\in [0,b)$, proving (ii) if $y_0=0$.   

We now prove (ii) when $y_0>0$. If there exists $t_0\in [0,b)$ such that $|z(t_0)|=y_0$, then 
using (i) and \eqref{nuc}, we find that 
$ \left| z\,h_3(w(y,z),y,z)\right|(t_0) <|c|z^2(t_0)$. 
Thus, \eqref{eq:zdeux} yields that $(z^2)'(t_0)<0$. This means that $|z(t)|=y_0$ has at most a solution in $[0,b)$. 
Hence, one of the following holds: 
\begin{itemize}
\item[(a)] $|z(t)|\leq y_0$ for all $t\in [0,b)$, which immediately yields (ii); 
\item[(b)] $|z(t)|\geq y_0$ for all $t\in [0,b)$;
\item[(c)] For some $t_0\in (0,b)$, we have 
$|z|> y_0$ on $t\in [0,t_0)$ and  
$|z|< y_0$ on $(t_0,b)$. 
\end{itemize}
Using \eqref{nuc} into \eqref{eq:zdeux}, we get $(z^2)'<0$ on $[0,b)$ in case (b) and on $[0,t_0)$ in case (c) since $\max\{y^2,z^2\}=z^2$. Thus 
in case (b) and (c) respectively, we find that $|z|\leq |z_0|$ on $ [0,b)$ and $[0,t_0)$, respectively. 
This proves (ii) when $y_0>0$. \qed 
\end{proof}

By (i), (ii) and the finite-time blow-up of solutions of ODEs, the flow $\Phi_t^w(y_0,z_0)$ associated to \eqref{syst:w} is defined for all $t\in [0,+\infty)$. Let $(y(t),z(t))$ be as in \eqref{yzx}.  

\begin{proof}[Proof of \eqref{ac}]
If $y_0=0$, then $y\equiv 0$ on $[0,\infty)$. Assuming $y_0>0$, then $y>0$ on $[0,\infty)$. The hypothesis on $h_2$ in \eqref{hy} implies that  
$(y^{1-p})'(t)\geq (p-1)C_1$ for all $t\geq 0$. By integration, we get that $\lim_{t\to +\infty}y(t)= 0$. Hence, for every $\varepsilon>0$, there exists
$t_\varepsilon>0$ large such that $0\leq y\leq \varepsilon$ on $[t_\varepsilon,\infty)$. 
To prove that $\lim_{t\to +\infty}z(t)= 0$, we show that there exists $\widetilde t_\varepsilon\geq t_\varepsilon$ such that 
$|z(t)|\leq \varepsilon$ for all $t\geq \widetilde t_\varepsilon$. Indeed, with a similar argument to the proof of (ii), it can be shown that
$|z(t)|=\varepsilon$ has at most one zero on $[t_\varepsilon,\infty)$. 
The option $|z|\geq \varepsilon$ on $[t_\varepsilon,\infty)$ is not viable here. Indeed, if $|z|\geq \varepsilon$ on $[t_\varepsilon,\infty)$, then 
again from \eqref{eq:zdeux} and \eqref{nuc}, we would have $(z^2)'\leq -|c| z^2\leq -|c|\varepsilon^2$ on $[t_\varepsilon,\infty)$, leading to 
a contradiction.
Hence, either $|z|\leq \varepsilon$ on $[t_\varepsilon,\infty)$ or 
there exists $\widetilde t_\varepsilon\in (t_\varepsilon,\infty)$ such that $|z|>\varepsilon$ on $[t_\varepsilon,\widetilde t_\varepsilon)$ and
$|z|<\varepsilon$ on $(\widetilde t_\varepsilon,\infty)$. In either of these cases, the conclusion $\lim_{t\to +\infty}z(t)= 0$ follows. This proves \eqref{ac}.  \qed
\end{proof}
The proof of Step~\ref{step71} is now complete. \qed
\end{proof}

\begin{step} \label{step72}
For any $\rho>0$, let $r_0\in (0,\delta/2)$ be as in Step~\ref{step71}  
and $3C_1(3+2C_2) r_0<\rho$. Then for any $w_j\in \mathcal X$ and  $\big(y_{0}^{(j)},z_{0}^{(j)}\big)\in D_{r_0}$ with $j=1,2$, we have 
\begin{equation}
\left| (y_1,z_1)-(y_2,z_2) \right|(t)\leq  e^{\rho t}\Big(\| w_1-w_2\|_\infty+
\big|(y_{0}^{(1)},z_{0}^{(1)})-(y_{0}^{(2)},z_{0}^{(2)})\big|\Big)
\end{equation}
for all $t\in [0,\infty)$, where we denote 
$(y_j(t),z_j(t)):=\Phi_t^{w_j}(y_{0}^{(j)},z_{0}^{(j)})$ for $j=1,2$. 
\end{step}

\begin{proof}[Proof of Step~\ref{step72}]
We denote $Y:=y_1-y_2$ and $Z:=z_1-z_2$. It suffices to prove that
\begin{equation} \label{suf}
e^{-2\rho t} (Y^2+Z^2)(t)
\leq \|w_1-w_2\|_\infty^2+(Y^2+Z^2)(0)\quad \text{for all } t\geq 0. 
\end{equation}
When clear, we drop the dependence on $t$ in notation. For $j=1,2$, we set
\begin{equation} \label{Ldef} P_j:=(w_j(y_j,z_j),y_j,z_j)  \ \text{and } 
 L:= Y\left[h_2(P_1)-h_2(P_2)\right]+Z\left[h_3(P_1)-h_3(P_2)\right].
 \end{equation}
By a simple calculation, we see that
\begin{equation} \label{hop}
\left(e^{-2\rho t}(Y^2+Z^2)\right)^\prime=2 e^{-2\rho t}\left[
-\rho \(Y^2+Z^2\)-\left|c\right|Z^2+L\right].
\end{equation}
We show that $L$ in \eqref{Ldef} satisfies
\begin{equation} \label{esti}
|L|\leq 3C_1 r_0\left[(3+2C_2) (Y^2+Z^2)+
\|w_1-w_2\|^2_\infty\right].
\end{equation}
\begin{proof}[Proof of \eqref{esti}] Since $\max\{|y_j|,|z_j|, |w_j(y_j,z_j)|\}\leq r_0$ for $j=1,2$, by
	the assumption on $|\nabla h_2|$ and $|\nabla h_3|$ in \eqref{hy}, we infer that 
\begin{equation} \label{lin} \sup_{\xi\in [0,1]}|(\nabla \phi)(\xi P_1+(1-\xi)P_2)| \leq 3C_1 r_0 \end{equation} with $\phi=h_2$ and $\phi=h_3$. Therefore, we get that 
\begin{equation} \label{ha1} |L|\leq 3C_1r_0 |P_1-P_2| \left(|Y|+|Z|\right)\leq 3\sqrt{2} C_1 r_0 |P_1-P_2| \sqrt{Y^2+Z^2}.\end{equation}
Set $a_1=\|w_1-w_2\|_\infty $ and $a_2=\sqrt{Y^2+Z^2}$. 
Using \eqref{Ldef} and that $w_1$ is $C_2$-Lipschitz, we get
\begin{equation} \label{ha2} \begin{aligned}|P_1-P_2|& \leq |w_1(y_1,z_1)-w_2(y_2,z_2)|+a_2 \\
&\leq |w_1(y_1,z_1)-w_1(y_2,z_2)|+a_1+ 
a_2\leq (1+C_2)a_2+a_1.
\end{aligned}\end{equation}
Plugging \eqref{ha2} into \eqref{ha1}, then using the inequality
$2a_1 a_2\leq a_1^2+a_2^2$, we conclude \eqref{esti}. \qed
\end{proof}
 
Using \eqref{esti} into \eqref{hop}, we get that
\begin{equation} \label{mio}
\left(e^{-2\rho t}\(Y^2+Z^2\)(t)\right)^\prime
\leq 2 e^{-2\rho t} \left[
\alpha_0 (Y^2+Z^2)(t)+ 3C_1r_0
\| w_1-w_2\|_\infty^2
\right],
\end{equation}
where $\alpha_0:= 3C_1(3+2C_2) r_0-\rho$ is negative from our choice of $r_0$. Hence, from \eqref{mio}, for every $t\in (0,\infty)$, we deduce that 
\begin{equation} \label{yoo}  \left(e^{-2\rho t}\(Y^2+Z^2\)(t)\right)^\prime\leq 
2\rho\, e^{-2\rho t} \| w_1-w_2\|_\infty^2.
\end{equation}
By integrating \eqref{yoo}, we obtain \eqref{suf}, 
which completes Step \ref{step72}. \qed \end{proof}

\begin{step} \label{step73}
Let $r_0\in (0,\delta/2)$ be as in Step~\ref{step72} with $\rho=a/2$ and $6C_1(1+C_2)r_0<a C_2$.      
Then, the map $T: \mathcal X\to \mathcal X$ is well-defined, where for every $w\in \mathcal X$, we put
$$Tw\(y_0,z_0\):=-\int_0^\infty e^{-a t}h_1\left(w\left(\Phi_t^w(y_0,z_0)\right), \Phi_t^w(y_0,z_0)\right)\, dt\ \ \text{for all }(y_0,z_0)\in D_{r_0}. $$
 \end{step}

\begin{proof}[Proof of Step~\ref{step73}] For all $(y_0,z_0)\in D_{r_0}$, we define $(y(t),z(t)):=\Phi_t^w(y_0,z_0)$. We now observe that for all $t\geq 0$, 
$h_1(w(y(t),z(t)),y(t),z(t))$ stays bounded
since $\max\{|w(y(t),z(t))|,|y(t)|,|z(t)|\}\leq r_0$. Then, $Tw\(y_0,z_0\)$ is well-defined since $a>0$. From $w(0,0)=0$, we have $\Phi_t^w(0,0)=(0,0)$ for all $t\geq 0$, which yields $Tw \(0,0\)=0$. To prove that $Tw\in \mathcal X$, it remains to show that $Tw$ ranges in $[-r_0, r_0]$ and $Tw$ is $C_2$-Lipschitz. 
Indeed, using \eqref{hy}, for every $\(y_0,z_0\)\in D_{r_0}$, we find that 
\begin{equation}
|Tw\(y_0,z_0\)|\leq  C_1 \int_0^\infty e^{-at}(w^2(y,z)+y^2+z^2)\, dt
\leq \frac{3C_1r_0^2}{a}.
\end{equation} 
Since $3C_1r_0<a$, we have $|Tw\(y_0,z_0\)|\leq r_0$ so that 
$Tw$ ranges in $[-r_0,r_0]$.

We prove that $Tw$ is $C_2$-Lipschitz. We fix $(y_{0}^{(j)},z_{0}^{(j)})\in D_{r_0}$ for $j=1,2$, then define $(y_j(t),z_j(t)):=\Phi_t^{w}\big(y_{0}^{(j)},z_{0}^{(j)}\big)$ and 
$P_j(t):=(w(y_j,z_j),y_j,z_j)(t)$ for all $t\geq 0$. By the definition of $Tw$, we see that
\begin{equation} \label{hp}
\big| Tw\big(y_{0}^{(1)},z_{0}^{(1)}\big)-Tw\big(y_{0}^{(2)},z_{0}^{(2)}\big)\big|  \leq \int_0^\infty e^{-at}\left|h_1(P_1)-h_1(P_2)\right|\, dt.
\end{equation}
Since \eqref{lin} holds for $\phi=h_1$, using $w_1=w_2=w$ in \eqref{ha2}, we get that
$$  \left|h_1(P_1)-h_2(P_2)\right|  \leq  3 C_1(1+C_2)r_0\left| (y_1,z_1)-(y_2,z_2) \right|. $$
Using that $6C_1(1+C_2)r_0<aC_2$ and taking $\rho=a/2$ in Step~\ref{step72}, we arrive at 
\begin{equation} \label{hp2}  \left|h_1(P_1)-h_1(P_2)\right|  \leq  \frac{aC_2}{2} e^{\frac{a t}{2}} \big| \big(y_0^{(1)},z_0^{(1)}\big)-\big(y_0^{(2)},z_0^{(2)}\big)\big|. 
\end{equation}  
From \eqref{hp} and \eqref{hp2}, we see that $Tw$ is $C_2$-Lipschitz, completing Step~\ref{step73}. \qed \end{proof}

\begin{step} \label{step74} If also $12C_1 r_0(2+C_2)<a$ in Step~\ref{step73},  then $T$ is a contraction on $\mathcal X$. \end{step}

\begin{proof}[Proof of Step~\ref{step74}]
For $w_1,w_2\in \mathcal X$ and $(y_0,z_0)\in D_{r_0}$, we define
$$ (y_j(t),z_j(t)):=\Phi_t^{w_j}(y_0,z_0) \quad\text{and} \quad P_j(t):=(w_j(y_j(t),z_j(t)),y_j(t),z_j(t)) $$
for all $t\geq 0$ and $j=1,2$. As in Step~\ref{step72} with $\rho=a/2$, we 
obtain that
\begin{equation} \label{tu1}
\begin{aligned}
|h_1(P_1)-h_1(P_2)| &\leq 3C_1 r_0 |P_1-P_2|
 \leq 
3C_1 r_0\left[(1+C_2) |(y_1,z_1)-(y_2,z_2)|+\|w_1-w_2\|_\infty\right]\\
&\leq 3C_1 r_0 (2+C_2) e^{\frac{at}{2}}\|w_1-w_2\|_\infty . 
\end{aligned}
\end{equation}
Then, using \eqref{tu1} and our choice of $r_0$, we see that
$$ 
\left| (Tw_1-Tw_2)(y_0,z_0)\right|\leq  
\int_0^\infty e^{-at} |h_1(P_1)-h_1(P_2)|\,dt<\frac{1}{2}\|w_1-w_2\|_\infty.
 $$
Therefore, $T$ is $1/2$-Lipschitz, so it is a contraction mapping. This ends Step \ref{step74}. \qed
\end{proof}

\begin{step} \label{step75} Let $r_0\in (0,\delta/2)$ be as in Step~\ref{step74}. Then, 
there exists $w\in \mathcal X$ such that 
for every $(y_0,z_0)\in D_{r_0}$ and $x_0=w(y_0,z_0)$, the initial value system \eqref{syst:th:bis} has a solution $\vec {\mathcal Z}=(x,y,z)$ defined on $[0,\infty)$ and satisfying \eqref{mn}. \end{step}

\begin{proof}[Proof of Step~\ref{step75}] The choice of $r_0$\bibliographystyle{spmpsci}  in Step~\ref{step74} depends only on $a,|c|, C_1, C_2$. Picard's fixed point theorem yields the existence of $w\in \mathcal X$ such that $Tw=w$, where $T$ is given by Step~\ref{step73}, that is,  
\begin{equation}\label{eq:w}
w(y,z)=-\int_0^\infty e^{-a\xi}h_1\left(w\left(\Phi_\xi^w(y,z)\right), \Phi_\xi ^w(y,z)\right)\, d\xi
\end{equation}
for all $(y,z)\in D_{r_0}$. We fix $(y_0,z_0)\in D_{r_0}$ arbitrarily. 
We show that $\vec {\mathcal Z}(t)=(x(t),y(t),z(t))$ is a solution to the initial system \eqref{syst:th:bis}, subject to \eqref{mn}, where 
we define \begin{equation} \label{defxyz} (y(t),z(t)):=\Phi_t^w(y_0,z_0)\quad \text{and}\quad  x(t):=w(y(t),z(t))\ \text{for all } t\geq 0.\end{equation} 
Indeed, Step~\ref{step71} yields that $(y',z')=(h_2(x,y,z),cz+h_3(x,y,z))$ on $ [0,\infty)$. In view of $w(0,0)=0$ and $\lim_{t\to +\infty}(y(t),z(t))= (0,0)$, we get $\lim_{t\to +\infty} x(t)=0$, proving \eqref{mn}. Since 
$$\Phi_\xi^w(y(t),z(t))=\Phi_\xi ^w\circ\Phi_t ^w(y_0,z_0)=\Phi_{t+\xi}^w(y_0,z_0)=(y(t+\xi),z(t+\xi))$$ for all $\xi,t\geq 0$, from \eqref{eq:w} and \eqref{defxyz}, we obtain that 
\begin{equation}\label{exp:x} \begin{aligned}
x(t)&= -\int_0^\infty e^{-a\xi}h_1\left(w\left(\Phi_\xi ^w(y(t),z(t))\right), \Phi_\xi ^w(y(t),z(t))\right)\, d\xi\\
&= -\int_0^\infty e^{-a\xi}h_1\left(w\left(y(t+\xi),z(t+\xi)\right), y(t+\xi),z(t+\xi)\right)\, d\xi\\
&= -e^{at}\int_t^\infty e^{-a \theta}h_1\left(w\left(y(\theta),z(\theta)\right), y(\theta),z(\theta)\right)\, d\theta 
\end{aligned}\end{equation} 
for all $t\geq 0$. 
This ends Step~\ref{step75} since 
$ x\in C^1[0,+\infty)$ and $x'=ax+h_1(x,y,z)$ on $[0,\infty)$. \qed
\end{proof}
Using the definition of $\mathcal X$ and Step~\ref{step75}, we finish the proof of Theorem~\ref{71}. \qed \end{proof}

\medskip\noindent{\bf Acknowledgements:} Work on this project started during first author's visits to McGill University (July--August 2017) and University of Lorraine (September--October 2017).
She is very grateful for the support and hospitality of her co-authors while carrying out research at their institutions. 



\end{document}